\documentclass[10pt]{amsart}
\usepackage{amsmath,latexsym}

\usepackage{color}

\definecolor{Myblue}{rgb}{0.0,0,0.9}
\definecolor{Mygreen}{rgb}{0.2,1,0}

\newtheorem{thm}{Theorem}[section]
\newtheorem{defthm}{Theorem}[section]
\newtheorem{lem}[thm]{Lemma}
\newtheorem{prop}[thm]{Proposition}
\newtheorem{cor}[thm]{Corollary}

\theoremstyle{definition}
\newtheorem{defn}[thm]{Definition}
\theoremstyle{definition}
\newtheorem{exmp}[thm]{Example}
\newtheorem{remark}[thm]{Remark}

\numberwithin{equation}{section}

\newcommand{\orho}{\overline{\rho}}

\newcommand{\benu}{\begin{enumerate}}
\newcommand{\enu}{\end{enumerate}}

\newcommand{\bema}{\left[\begin{array}}
\newcommand{\ema}{\end{array}\right]}

\newcommand{\gen}[1]{\langle#1\rangle}
\newcommand{\field}{k}
\newcommand{\LL}{\begin{picture}(15,5)\put(2,3){\line(1,0){10}}\end{picture}}
\newcommand{\xLL}[1]{\begin{picture}(15,12)\put(2,3){\line(1,0){10}}\put(7.5,6){\HBCenter{\small
      $#1$}}\end{picture}}

\newcommand{\CluCat}{\mathcal{C}}

\renewcommand{\mod}{\operatorname{mod}}
\newcommand{\op}{\operatorname{op}}
\newcommand{\add}{\operatorname{add}}
\newcommand{\End}{\operatorname{End}}

\newcommand{\Hom}{\operatorname{Hom}}
\newcommand{\Ext}{\operatorname{Ext}}
\newcommand{\Ker}{\operatorname{Ker}}
\newcommand{\Der}{\operatorname{D^b}}
\newcommand{\HomD}{\operatorname{Hom}_{\Der(H)}}
\newcommand{\HomC}{\operatorname{Hom}_{\CluCat}}
\newcommand{\EndD}{\operatorname{End}_{\Der(H)}}
\newcommand{\EndC}{\operatorname{End}_{\CluCat}}
\newcommand{\dual}{\mathrm{D}}
\newcommand{\ealg}{\mathcal{R}}

\newcommand{\Pp}{\mathcal{P}}
\newcommand{\Rr}{\mathcal{R}}

\newcommand{\Ii}{\mathcal{I}}
\newcommand{\rad}{\operatorname{rad}}
\newcommand{\soc}{\operatorname{soc}}
\newcommand{\TOP}{\operatorname{top}}

\newcommand{\gldim}{\operatorname{gldim}}

\newcommand{\Groth}[1]{\operatorname{K_\circ}(#1)}
\newcommand{\proj}{\mathbf{P}}
\newcommand{\inj}{\mathbf{I}}
\newcommand{\An}{\mathbb{A}}
\newcommand{\ZZ}{\mathbb{Z}}
\newcommand{\ra}{\rightarrow}
\newcommand{\oT}{\widetilde{T}}
\newcommand{\oB}{\widetilde{B}}
\newcommand{\id}{\operatorname{id}}
\newcommand{\trivpath}[1]{\operatorname{e}_{#1}}

\def\semidirprod{\begin{picture}(8,8)\qbezier(2,0.5)(5,3.5)(8,6.5)\qbezier(2,6.5)(5,3.5)(8,0.5)\put(2,0.5){\line(0,1){6}}\end{picture}}

\newcommand{\proofend}{\hfill$\Box$\par}

\newcommand{\HVCenter}[1]{\setbox 0=\hbox{#1}%
        \dimen0=\wd0%
        \dimen1=\ht0%
        \divide\dimen0 by 2%
        \divide\dimen1 by 2%
        \hskip -\dimen0%
        \lower \dimen1%
        \box0%
        \hskip -\dimen0}
\newcommand{\HBCenter}[1]{\setbox 0=\hbox{#1}%
        \dimen0=\wd0%
        \dimen1=\ht0%
        \divide\dimen0 by 2%
        \hskip -\dimen0%
        \box0%
        \hskip -\dimen0}
\newcommand{\LTCenter}[1]{\setbox 0=\hbox{#1}%
        \dimen1=\ht0%
        \lower \dimen1%
        \box0%
        \hskip -\dimen0}
\newcommand{\HTCenter}[1]{\setbox 0=\hbox{#1}%
        \dimen0=\wd0%
        \dimen1=\ht0%
        \divide\dimen0 by 2%
        \hskip -\dimen0%
        \lower \dimen1%
        \box0%
        \hskip -\dimen0}
\newcommand{\RTCenter}[1]{\setbox 0=\hbox{#1}%
        \dimen0=\wd0%
        \dimen1=\ht0%
        \hskip -\dimen0%
        \lower \dimen1%
        \box0%
        \hskip -\dimen0}
\newcommand{\RBCenter}[1]{\setbox 0=\hbox{#1}%
        \dimen0=\wd0%
        \dimen1=\ht0%
        \hskip -\dimen0%
        \box0%
        \hskip -\dimen0}
\newcommand{\RVCenter}[1]{\setbox 0=\hbox{#1}%
        \dimen0=\wd0%
        \dimen1=\ht0%
        \divide\dimen1 by 2%
        \hskip -\dimen0%
        \lower \dimen1%
        \box0%
        \hskip -\dimen0}
\newcommand{\LVCenter}[1]{\setbox 0=\hbox{#1}%
        \dimen1=\ht0%
        \divide\dimen1 by 2%
        \lower \dimen1%
        \box0%
        \hskip -\dimen0}

\newcommand{\mylabel}[1]{\put(0,0){\color{white} \circle*{7}}\put(0,0){\circle{7}}\put(0,0){\HVCenter{\tiny\bf #1}}}
\newcommand{\mylabeltwo}[2]{\put(0,0){\color{white} \circle*{7}}\put(0,0){\circle{7}}\put(0,0){\HVCenter{\tiny\bf #1}}\put(0,4.5){\HVCenter{\tiny #2}}}

\newcommand{\mypdott}{\put(0,0){\circle*{2}}}

\begin{document}
\sloppy

\title{From iterated
  tilted algebras to cluster-tilted algebras}

\author[Barot]{Michael Barot}
\address{Michael Barot, 
Instituto de matem\'aticas, Universidad Nacional Aut\'onoma
  de M\'exico, Ciudad Universitaria, C.P. 04510, D.F., Mexico.}
\email{barot@matem.unam.mx}

\author[Fern\'andez]{Elsa Fern\'andez}
\address{Elsa Fern\'andez, 
Facultad de Ingenier\'{\i}a, Universidad Nacional de la
  Patagonia San Juan Bosco, 9120 Puerto Madryn, Argentina.}
\email{elsafer9@gmail.com}

\author[Platzeck]{Mar\'{\i}a In\'es Platzeck}
\address{Mar\'{\i}a In\'es Platzeck, Instituto de
    Matem\'atica, Universidad Nacional del Sur, 8000, Bah\'{\i}a
    Blanca, Argentina.}
\email{platzeck@uns.edu.ar}

\author[Pratti]{Nilda Isabel Pratti} 
\address{Nilda Isabel Pratti,
  Departamento de Matem\'atica, Facultad de Ciencias Exactas y
  Naturales, Funes 3350, Universidad Nacional de Mar del Plata, 7600
  Mar del Plata, Argentina.}
\email{nilpratti@gmail.com}

\author[Trepode]{Sonia Trepode}
\address{Sonia Trepode, 
Departamento de Matem\'atica, Facultad de Ciencias Exactas y
  Naturales, Funes 3350, Universidad Nacional de Mar del Plata, 7600
  Mar del Plata, Argentina.}
\email{strepode@gmail.com}

\begin{abstract}
  In this paper the relationship between iterated tilted algebras and
  cluster-tilted algebras and relation-extensions is studied. In
  the Dynkin case, it is shown that the relationship is very strong
  and combinatorial.
\end{abstract}

\thanks{M.~I.~Platzeck and S.~Trepode are researchers from CONICET,
  Argentina. E.~Fern\'andez, M.~I.~Platzeck, N.~I.~Pratti and
  S.~Trepode thankfully acknowledge partial support from CONICET and
  from Universidad Nacional del Sur, Argentina.}

\maketitle

\section{Introduction and Results}

Cluster algebras were conceived around 2000 by Fomin and Zelevinsky,
see \cite{FZ1}, where they axiomatized a kind of combinatorics which
was rapidly recognized to have been present before in different areas.
Such a connection was established in the seminal paper \cite{BMRRT} to
the representation theory of finite-dimensional algebras, where the
authors introduced the concept of \emph{cluster category} $\CluCat$,
defined as orbit category of the bounded derived category $\Der(H)$ of
a finite-dimensional hereditary algebra $H$ over a field $k$.  They
established the connection in the special case when $k$ is
algebraically closed and $H$ is of finite representation type, that
is, the quiver of $H$ is the disjoint union of Dynkin diagrams. It is
remarkable that in the setting of cluster algebras the concept of
finite type also exists naturally and that it is given by the
Cartan-Killing classification, see \cite{FZ2}. The connection between
cluster algebras and cluster categories was deepened by various
authors and expanded over the original limit of finite type to
hereditary finite-dimensional algebras (over an algebraically closed
field) in general, see for example \cite{CK2}, \cite{BMRT}.

We assume throughout the whole article that the base field $k$ is
algebraically closed. The connection established thus far shows that
to each hereditary algebra $H$, a cluster algebra $\mathcal{A}$ can be
associated in such a way that its cluster variables (resp. clusters)
correspond precisely to the indecomposable rigid objects, that is,
objects $T$ with $\Hom_{\CluCat}(T,T[1])=0$ where $[1]$ is the shift
induced by the shift in $\Der(H)$ (respectively \emph{cluster-tilting
  objects}, see Section \ref{sec:clucat}) of the cluster category
$\CluCat$.  This turned the attention to \emph{cluster-tilted
  algebras}, that is, endomorphism algebras of cluster-tilting objects
of $\CluCat$, see \cite{BMR2,BMR3}. Buan, Marsh and Reiten showed in
\cite{BMR3} that the quivers of the cluster-tilted algebras arising
from a given cluster category are exactly the quivers corresponding to
the exchange matrices of the associated cluster algebra. Moreover,
they showed that for each cluster-tilting object $T=T'\oplus T_i$ with
indecomposable summands $T_i$ there exists precisely one
indecomposable object $T_i'\not\simeq T_i$ such that $T'\oplus T_i'$
is again a cluster-tilting object and that this procedure corresponds
in natural way to the mutation of the associated seeds.

In \cite{ABS} the authors studied the relationship between tilted
algebras $\End_H(M)$ for tilting $H$-modules $M$, and cluster-tilted
algebras $\End_{\CluCat}(T)$ for cluster-tilting objects $T$ in
$\CluCat$. For this they introduced the concept of \emph{relation
  extension} of an algebra $B$ with $\gldim B\leq 2$ and defined it to
be the algebra $\ealg(B)=B\semidirprod\Ext_B^2(\dual B,B)$, where
$\dual B$ is the dual of $B$, that is, the injective cogenerator
$\Hom_k(B,k)$ of the module category $\mod B$. They proved that an
algebra $C$ is a cluster-tilted algebra if and only if it is the
relation-extension of some tilted algebra $B$.  This result has an
analogy with a well known theorem about the relation between trivial
extensions $T(A) = A \semidirprod D(A)$ of artin algebras A and tilted
algebras, due to Hughes and Waschb\"ush \cite{HW}. They prove
that $T(A)$ is of finite representation type if and only if there
exists a tilted algebra $B$ of Dynkin type such that $ T(A) \simeq
T(B)$.  This connection was extended to iterated tilted algebras by
Assem, Happel and Rold\'an \cite{AHR}, who proved that a trivial
extension $T(A)$ is of finite representation type if and only if $A$
is an iterated tilted algebra of Dynkin type.  Keeping these results
in mind, we want to further extend the mentioned connection between
cluster tilted algebras and tilted algebras to iterated tilted
algebras. It turns out that it is possible to do so, but one needs to
restrict to iterated tilted algebras of global dimension at most
two. The following is one of our main results.

\begin{thm}
  \label{thm:split}
  If $B$ is an iterated tilted algebra of $\gldim B\leq 2$ then there
  exists a cluster-tilted algebra $C$ which is a split extension of
  $B$. More precisely, if $B=\EndD(T)$ with $H$ a hereditary algebra
  and $T$ is a tilting complex in $\Der(H)$ then
  $C=\End_{\CluCat(H)}(T)$ is a cluster-tilted algebra and there
  exists a sequence of algebra homomorphisms
  $$
  B\ra C\xrightarrow{\pi}\ealg(B)\ra B
  $$
  whose composition is the identity map. Moreover, the kernel of $\pi$
  is contained in $\rad^2 C$. In particular $C$ and $\ealg(B)$ have
  the same quivers and are both split extensions of $B$.
\end{thm}

The last assertion, relating the quivers of $C$ and $\ealg(B)$, was
also proven independently by Amiot in \cite [4.17]{A} with different
thecniques.

To achieve the result we introduce a mechanism of obtaining a new
iterated tilted algebra $\rho(B)$ with $\gldim \rho(B)\leq 2$, from a
given one $B$ with $\gldim B\leq 2$. We shall call the new algebra
$\rho(B)$ the \emph{rolling} of $B$. The key result in our proof is
the following.
 
\begin{thm}\label{thm:rolling-to-tilted}
  Let $B$ be an iterated tilted algebra of type $Q$ with $\gldim B\leq
  2$ then for sufficiently large $h$ the algebra $\rho^h(B)$ is tilted
  of type $Q$.
\end{thm}

We then focus on the finite type, where much more precise information
is available on the combinatorial structure of the quiver and
relations of a cluster-tilted algebra, see \cite{BMR2}. To do this
we need the notion of admissible cut of a quiver $Q$, introduced in
\cite{F} (see also \cite{FP}), and define it to be a subset $\Delta$
of the arrows such that each oriented chordless cycle of $Q$ contains
precisely one element of $\Delta$. Then for an algebra $B$, given as
the quotient of a path algebra $\field Q_B$ by an admissible ideal
$I_B$, we define the \emph{quotient} of $B$ \emph{by an admissible
  cut} $\Delta$ to be $\field Q_B/\langle I_B\cup \Delta\rangle$.
 
The following shows that the relationship between cluster-tilted
algebras and iterated tilted algebras of the same type is strong and
combinatorial.

\begin{thm}\label{thm:iff-cut}
  An algebra $B$ with $\gldim B\leq 2$ is iterated tilted of Dynkin type $Q$
  if and only if it is the quotient of a cluster-tilted algebra of type $Q$
  by an admissible cut.
\end{thm}

Moreover, we characterize the iterated tilted algebras $B$ with
$\gldim B\leq 2$ for which the relation extension $\ealg(B)$ is
isomorphic to the corresponding cluster-tilted algebra $C(B)$,
see Proposition \ref{prop:char-r=c}.

Results along these lines were proven in \cite{F} and \cite{FP} for
admissible cuts of trivial extensions.  In her PhD thesis
E. Fern\'andez showed that they are a very useful tool in the study of
classification problems. In this way, she classified all trivial
extensions of finite representation type, and gave a method to get all
iterated tilted algebras of Dynkin type obtaining, under a unified
approach, results proven with diverse techniques by other authors.
Though in a different context, we consider that the results in this
paper can be applied in a similar way to obtain analogous
classification results for cluster tilted algebras and also provide a
new insight on tilted an iterated tilted algebras to study their
quivers and relations.

\section{Basic definitions and notations}

\subsection{Quivers and path algebras}
\label{subsec:quivers}
A \emph{quiver} is a directed graph, that is, a quadruple
$Q=(Q_0,Q_1,s,t)$, where $Q_0$ is the set of vertices, $Q_1$ the set
of arrows and $s,t\colon Q_1\ra Q_0$ are the maps which assign to each arrow
$\alpha$ its \emph{source} $s(\alpha)$ and its \emph{target}
$t(\alpha)$. We usually write $\alpha\colon s(\alpha)\ra t(\alpha)$ to
express this.

A subquiver $Q'$ of a quiver $Q$ is called a \emph{chordless} (or
\emph{minimal}) \emph{cycle} if $Q'$ is full, connected and in every
vertex of $Q'$ exactly two arrows of $Q'$ incide (starting or stopping
there).  In case exactly one arrow stops and the other starts the
cycle is called \emph{oriented}.

A \emph{path} is a tuple
$\gamma=(y|\alpha_r,\alpha_{r-1},\ldots,\alpha_1|x)$
of vertices $x,y\in Q_0$ and arrows $\alpha_1,\ldots,\alpha_r\in Q_1$
with $x=y$ if $r=0$ and $s(\alpha_1)=x$, $t(\alpha_r)=y$,
$t(\alpha_i)=s(\alpha_{i+1})$ for $i=1,\ldots,r-1$ if $r>0$.
The number $r$ is called the \emph{length} of $\gamma$ and the
functions $t,s$ are naturally extended by setting $s(\gamma)=x$ and
$t(\gamma)=y$. We usually abbreviate
$(y|\alpha_r,\alpha_{r-1},\ldots,\alpha_1|x)$ by $\alpha_r\alpha_{r-1}\cdots\alpha_1$ and $(x||x)$ by $\trivpath{x}$.

For a field $\field$ and a quiver $Q$, let $\field Q$ be the
\emph{path algebra} of $Q$: the underlying $\field$-vector space has
the set of all paths as basis and the multiplication is induced
linearly by the concatenation of paths, that is, if
$\delta=\beta_s\cdots\beta_1$ and $\gamma=\alpha_r\cdots\alpha_1$ then
$\delta\gamma$ is defined as
$$
\delta\gamma=\beta_s\cdots\beta_1\alpha_r\cdots\alpha_1
$$
if $s(\beta_1)=t(\alpha_r)$ and $\delta\gamma=0$ otherwise.
The ideal of $\field Q$ generated by all paths of positive length is
called \emph{radical} and will be denoted by $\rad \field Q$.

If the field $\field$ is algebraically closed, then each
finite-dimensional algebra $A$ is Morita-equivalent to the quotient of
a path-algebra by an \emph{admissible} ideal $I$, that is, $I$ is
contained in $\rad^2 \field Q$ and the quotient $\field Q/I$ is
finite-dimensional. If, moreover, $A$ is basic then $A\simeq kQ/I$,
and the pair $(Q,I)$ is called a \emph{presentation} for $A$.
If $Q$, $Q'$ are two quivers and $I\subset k Q$, $I'\subset k Q'$
  two ideals then we call $(Q',I')$ an \emph{extension of} $(Q,I)$ if
  $Q_0\subseteq Q'_0$, $Q_1\subseteq Q'_1$ and $I\subseteq I'$.
 
\subsection{Split extensions}

We say that the algebra $A$ is a \emph{split extension}
of the algebra $B$ by the ideal $M$ of $A$
if there exists a split surjective algebra morphism $\pi\colon A
\longrightarrow B$ whose kernel $M$ is a nilpotent ideal. This
means that there exists a short exact sequence of $\field$-vector spaces
$$ 
0 \longrightarrow M \stackrel{l} \longrightarrow A
\stackrel{\pi}\longrightarrow B \longrightarrow O
$$ 
such that there exists an algebra morphism
$\sigma\colon B \longrightarrow A$ with $\pi \sigma = 1_B$. In
particular $\sigma$ identifies $B$ with a subalgebra of $A$. 
Note that $M \subseteq\rad A$ since $M$ is a nilpotent ideal.

Let $B$ be a finite dimensional algebra  and
consider a $B$-$B$-bimodule $M$. The \emph{trivial extension} $B
\semidirprod M$ is the algebra whose underlying $\field$-vector space is $B
\times M$ with multiplication $(b,m)\cdot(b',m')=(bb',b m'+m b')$. When $\gldim B\leq 2$, the
trivial extension $\ealg(B)=B\semidirprod \Ext^2_B(DB,B)$ is called
the \emph{relation extension} of $B$, see \cite{ABS}.

\subsection{Quadratic forms}
\label{sec:qf}

For an algebra of finite global dimension $B$, we denote by $\mod B$
the category of finitely generated (or equivalently
finite-dimensional) left $B$-modules. Furthermore, we denote by
$\Groth{B}$ the associated \emph{Grothendieck group}, that is, the
free abelian group on the isomorphism classes of objects of $\mod B$
modulo the subgroup generated by $\{E-X-Y\mid 0\ra X\ra E\ra Y\ra
0\text{ is exact }\}$. The class of a $B$-module $X$ shall be denoted
by $[X]$. Notice that $\Groth{B}\simeq \ZZ^n$ where $n$ is the number
of isomorphism classes of simple $B$-modules. We denote by
$\chi_B\colon \Groth{B}\rightarrow \ZZ$ the \emph{homological form}
(or \emph{Euler form}) of $B$, that is, $\chi_B$ is the quadratic form
associated to the bilinear form defined by
$$
([X],[Y])=\sum_{i=0}^{\infty}\dim\Ext_B^i(X,Y)
$$
for $X,Y \in$ mod$B$.
    
We denote by $q_B$ the \emph{geometrical form} (or \emph{Tits form}),
defined by the ``truncated'' bilinear form defined for the classes of
the simple modules $S_i$ by
$$
\langle[S_h],[S_j]\rangle=\sum_{i=0}^{2}\dim\Ext_B^i(S_h,S_j).
$$

\begin{remark} \label{rem:q=chi}
  If $\gldim B\leq 2$ then $\chi_B=q_B$.
\end{remark}

\subsection{Algebras which are simply connected}
\label{sec:sc}
An algebra $A$ with connected quiver $Q$ with no oriented cycles is
called \emph{simply connected} if for each presentation $(Q,I)$ of $A$
the fundamental group $\pi(Q,I)$ is trivial, for precise
  definitions we refer to \cite{BG} and \cite{Sko}.

A full subquiver $Q'$ of $Q$ is called \emph{convex} if for any two
paths $\gamma$, $\delta$ with $t(\gamma)=s(\delta)$ and $s(\gamma),
t(\delta)\in Q'_0$ then $t(\gamma)\in Q'_0$.  An algebra $A = \field
Q/I$ is called \emph{strongly simply connected} if for every full and
convex subquiver $Q'$ of $Q$ the induced algebra $\field Q'/(\field
Q'\cap I)$ is simply connected.

\begin{remark}
\label{rem:sc=ssc}
By \cite[Def. 2.2]{Sko} and \cite[2.9]{BG}, if $A$ is of finite representation type then $A$ is simply connected if and only if it is strongly simply connected.
\end{remark}

\subsection{Tilted and iterated tilted algebras}

Let $A$ be a finite-dimensional $k$-algebra. 
We recall that a module $M\in\mod A$ is called \emph{tilting
  module} if $M$ has projective dimension at most one, $\Ext_A^1(M,M)=0$ and
the decomposition of $M$ into indecomposables contains precisely $n$
pairwise non-isomorphic summands, where $n$ is the number of pairwise
non-isomorphic simple $A$-modules, or equivalently the number of
vertices of the quiver of $A$.

If $H$ is a hereditary algebra and $M$ a tilting $H$-module then
$\End^{\op}_H(M)$ is called a \emph{tilted algebra}. Since the
opposite of a tilted algebra is again a tilted algebra we often prefer
to look at the endomorphism algebras themselves instead of their
opposites. An algebra $B$ is called an \emph{iterated tilted algebra
  of type $Q$} if there exists a sequence of algebras
$A_1,A_2,\ldots,A_t$ such that $A_1$ is hereditary with quiver $Q$,
$A_t=B$ and for each $i=1,\ldots,t-1$ we have
$A_{i+1}\simeq \End_{A_i}(M_i)$ for some tilting $A_i$-module $M_i$ or
$A_i\simeq \End_{A_{i+1}}(N_i)$ for some tilting $A_{i+1}$-module
$N_i$.
\subsection{Structure of the derived category over a hereditary algebra}
\label{sec:der-struc}
Throughout the rest of the article $H$ denotes a finite-dimensional
hereditary algebra over an algebraically closed field $k$. We denote by $\Der(H)$ the bounded derived
category of finitely generated $H$-modules, see \cite{Ha} for
generalities on derived categories. Since $H$ is hereditary, each
indecomposable object of $\Der(H)$ is isomorphic to a complex
concentrated in one degree. We shall identify the objects in $\mod H$
with the complexes concentrated in degree zero.

Recall that in $\Der(H)$ \emph{Serre duality} holds, that is, for any
objects $X$ and $Y$ of $\Der(H)$, we have
$$
\HomD(X,\tau Y)=\dual\Hom(Y,X[1]),
$$
where $\tau$ denotes the Auslander-Reiten translation and
$[1]$ the suspension in $\Der(H)$. The autoequivalence
$F=\tau^{-1}\circ [1]$ will play a crucial role in the rest of the
paper. 

If the quiver $Q$ of $H$ is Dynkin then the Auslander-Reiten quiver
$\Gamma$ of $\Der(H)$ consists of a single transjective component
isomorphic to the translation quiver $\ZZ Q$, see \cite[Ch.1,
Cor. 5.6]{Ha}. In particular, the arrows induce a partial order in the
vertices of $\Gamma$, that is, if $L\ra M$ is an arrow in $\Gamma$
then we write $L<M$. Moreover if there exists a path from $L$ to $M$
then all paths have the same length $d(L,M)$ and we set $d(L,M)=0$ if
there is no path at all.

In case $Q$ is Dynkin, a set of representatives
$\Sigma_1,\ldots,\Sigma_n$ of the $\tau$-orbits of $\Gamma$ is called
\emph{section} if $\Sigma_1,\ldots,\Sigma_n$ induce a connected subquiver of
$\Gamma$. Here $n$ is the the number of vertices in the quiver $Q$.

If the quiver $Q$ of $H$ is not Dynkin then the structure of the
Auslander-Reiten quiver $\Gamma$ of $\Der(H)$ is completely different.
Denote by $\Pp$, (resp.  $\Ii$) the preprojective (resp. preinjective)
component of the Auslander-Reiten quiver of $H$ and by $\Rr$ the full
subcategory of $\mod H$ given by the regular components. For each
$r\in \ZZ$ the regular part $\Rr$ gives rise to $\Rr[r]$, given by the
complexes $X\in \Der(H)$ concentrated in degree $r$ with
$X_r\in\Rr$. Moreover, for each $r\in\ZZ$ there is a transjective
component $\Ii[r-1]\vee\Pp[r]$ of $\Gamma$ which we shall denote by
$\Rr[r-\frac{1}{2}]$, and each component of $\Gamma$ is contained in
$\Rr[r]$ for some half-integer $r$.  The notation has the advantage
that the different parts are ordered in the sense that
$\Hom(\Rr[a],\Rr[b])=0$ for any two half-integers $a>b$. Also note
that $\Hom(\Rr[a],\Rr[b])=0$ if $a<b-1$.

\subsection{Tilting complexes}
\label{sec:tilt-cmplx}
An object $T$ of $\Der(H)$ is called \emph{tilting complex} if
$\Hom(T,T[i])=0$ for each $i\neq 0$ and if the only object $X$ for
which $\Hom(T,X[i])=0$ for all $i$ is the zero object. It follows from
\cite[Cor. 3.3 and Lemma 3.5]{RvB} that $T$ is a tilting complex if
and only if $\HomD(T,T[i])=0$ for all $i\neq 0$ and $T$ has exactly
$n$ non-isomorphic indecomposable summands, where $n$ is the number of
simple $H$-modules (up to isomorphism).

Note that by \cite[Cor. 5.5 of Chap. 4]{Ha} and \cite{Ric}, 
an algebra $A$ is iterated tilted of type
$Q$ if and only if $A$ is isomorphic to the endomorphism algebra of a
tilting complex $T$ in $\Der(\field Q)$ (equivalently if and only if
there exists an equivalence of triangulated categories $\Der(A)\simeq
\Der(\field Q)$).

\subsection{The cluster category}
\label{sec:clucat}
Let $H$ be a hereditary algebra. Then the orbit category
$\CluCat=\Der(H)/F^\ZZ$ is called \emph{cluster category of $H$}, see
\cite{BMRRT}. By construction the objects of $\CluCat$ are the objects
of $\Der(H)$ and the morphism spaces are given by
$$
\HomC(X,Y)=\bigoplus_{i\in \ZZ}\HomD(X,F^i Y)
$$
with the natural composition, see \cite{Ke}, where it is also shown
that $\CluCat$ is a triangulated category. 

An object $T$ of $\CluCat$ is a \emph{cluster-tilting object} if
$\Hom(T,T[1])=0$ and if $T$ is decomposed into indecomposables
$T=\bigoplus_{i=1}^n T_i$ then there are precisely $n$ pairwise
non-isomorphic summands, where $n$ is the number of simple
$H$-modules. 

\section{Iterated tilted algebras of global dimension two}
\label{sec:roll}

\subsection{Generalities on tilting complexes} 
\label{sec:tilt-basic}
If $T$ is a tilting complex in $\Der (H)$ (see Section \ref{sec:tilt-cmplx}) and
$B=\EndD(T)$ then we have an equivalence of categories $G\colon
\Der(H)\ra\Der(B)$ derived from $\Hom(T,-)$ such that $G(T)=B$ and
$G(\tau T[1])=\dual B$. For any direct summand $X$ of $T$ we write
$$
\proj_{X,T}=G(X)=\HomD(T,X)\quad\text{ and }\quad
\inj_{X,T}=G(\tau X[1])
$$ Moreover, if $GX$ and $GY$ are $B$-modules, for two objects $X$ and
$Y$ of $\Der(H)$, then $\Ext^i_B(GX,GY)\simeq \HomD(X,Y[i])$ for all
$i\in \ZZ$.

\begin{lem}\label{lem:basic}
  Let $T$ be a tilting complex in $\Der(H)$ such that $\gldim B\leq
  2$, where $B=\EndD(T)$. Then $\Hom_{\Der(H)}(T,F^{-1}T)=0$ and
  $\Hom_{\Der(H)}(T,F^{-2}T)=0$.
\end{lem}

\begin{proof}
  By Serre duality and the fact that $T$ is a tilting complex 
  we have $\HomD(T,F^{-1}T)=\HomD(T[1],\tau T)=\dual\HomD(T,T[2])=0$.
  Also
  $\HomD(T,F^{-2}T)=\HomD(T[3],\tau^2 T[1])=\dual\HomD(\tau
  T[1],T[4])=\Ext_B^4(\dual B,B)=0$ again by Serre duality and $\gldim
  B\leq 2$.
\end{proof}

If $T$ is a tilting complex then we have as in \cite{ABS}
that $\Ext_B^2(\dual B,B)\simeq\HomD(\tau
T[1],T[2])\simeq\HomD(F^{-1}T,T)\simeq\HomD(T,FT)$ with the natural
structure of $B$-$B$-bimodules. 

\subsection{The rolling of tilting complexes}
\label{sec:rolling}

We are now going to define a procedure which is important in the
forthcoming. It defines for each tilting complex a new complex
$\rho(T)$ such that $T\simeq \rho(T)$ in the cluster category
$\CluCat$.  Since the structure of the derived category $\Der(H)$ is
substantially different whether the quiver $Q$ of $H$ is Dynkin or
not, we have to distinguish these two cases in the construction.
 
Let first $Q$ be a Dynkin quiver and $T$ a tilting complex of
$\Der(\field Q)$.  Since $T=\bigoplus_{i=1}^n T_i$ has only finitely
many summands we can easily find a section
$\Sigma=\{\Sigma_1,\ldots,\Sigma_n\}$ such that $T\leq \Sigma$, that
is, $T_i\leq \Sigma_j$ for all $i$ and $j$. If $\Sigma_j$ is maximal
in $\Sigma$ and $\Sigma_j\not\in\{T_1,\ldots,T_n\}$ then
$\Sigma'=\Sigma\setminus\{\Sigma_j\}\cup\{\tau\Sigma_j\}$ is also a
section satisfying $T\leq \Sigma'$. After finitely many steps we get a
section $\Sigma(T)$ such that $T\leq \Sigma(T)$ and all maximal
elements in $\Sigma(T)$ belong to add$T$. Notice that the section
$\Sigma(T)$ is uniquely defined by $T$.
 
\begin{defn}[Rolling of tilting complex, the Dynkin case]
  \label{def:rolling-D}
With the previous notations, let $X$ be the sum of those summands of $T$
which belong to $\Sigma(T)$ and $T'$ a complement of $X$ in $T$. Then
define the \emph{rolling} of $T$ to be $\rho(T)=T'\oplus F^{-1} X$.
\end{defn}

Now consider the case where $Q$ is not Dynkin. Recall from Section
\ref{sec:der-struc} that $\Der(\field Q)$ is composed by the parts
$\Rr[r]$ for $r\in\ZZ/2$ where $\Rr[r]$ denotes the regular
(resp. transjective) part if $r$ is an integer (resp. not an integer).
Now, write $T=\bigoplus_{a\in\ZZ/2}T_{\Rr[a]}$, where
$T_{\Rr[a]}\in\Rr[a]$.

\begin{defn}[Rolling of tilting complex, the non-Dynkin case]
  \label{def:rolling-nD}
With the previous notation let $m$ be the largest half-integer such
that $T_{\Rr[m]}$ is non-zero. Then define $X=T_{\Rr[m]}$ and $T'$ to
be the complement of $X$ in $T$. Define the \emph{rolling} of $T$ to
be $\rho(T)=T'\oplus F^{-1} X$.
\end{defn}

\begin{remark}\label{rem:hom-X-T'}
  If $T=T'\oplus X$ is a tilting complex in $\Der(H)$ and
  $\rho(T)=T'\oplus F^{-1} X$ then we have $\HomD(X,T')=0$.
\end{remark}

\begin{defn}[Rolling of iterated tilted algebras]
\label{rolling of B}
  Let $B$ be an iterated tilted algebra. Then define $\rho(B)$ to be
  the endomorphism algebra $\EndD(\rho(T))$, where $H$ is a hereditary
  algebra with $\Der(B)\simeq \Der(H)$ and $T$ a tilting complex in
  $\Der(H)$ with $B=\EndD(T)$.
\end{defn}

Notice that $\rho(B)$ does not depend on the choice of $H$ or $T$. In
fact, if $T$ and $\hat T$ are tilting complexes in $\Der(H)$ such that
$\End_{\Der(H)}(T) \simeq \End_{\Der(H)}(\hat T)$ then there is an
equivalence of categories $G\colon \Der(H) \ra \Der(H)$ with $G(T)=\hat T$,
and $G$ preserves the partial order in $\Der(H)$. Thus in the Dynkin
case $G(\Sigma(T)) \simeq \Sigma(\hat T)$, and the sum $X$ of the
maximal elements in $\Sigma (T)$ corresponds under $G$ to the sum
$\hat X$ of the maximal elements in $\Sigma(\hat T)$. Thus $\rho(T)$
and $G(\rho(T))\simeq \rho(\hat T)$ have isomorphic endomorphism
rings. The argument in the non-Dynkin case is similar.

\subsection{Characterization when $\rho(T)$ is again a tilting complex}

The following results provide necessary and sufficient conditions for
the rolling $\rho(T)$ to be a tilting complex again.

\begin{lem}\label{lem:crit}
  Let $T = T^{\prime} \oplus X$ be a tilting complex in $\Der(H)$ such
  that $\Hom_{\Der(H)}(X,T') = 0$ and let $B=\End_{\Der(H)} (T)$.
  Then $\overline{T} = T^{\prime} \oplus F^{-1}X$ is a tilting complex
  if and only if $\Hom_{\Der(H)} (F^{-1}X,T'[j])=0$ for all $j\neq 0$
  if and only if $\Ext^j_B(\inj_{X,T},\proj_{T',T}) = 0$ for each
  $j\neq 2$ .
\end{lem}

\begin{proof}
  Observe that $\Hom_{\Der(H)} (T', F^{-1}X[j]) = \Hom_{\Der(H)} (T',
  \tau X[j-1]) = \Hom_{\Der(H)}(X[j-1],T'[1]) = \Hom_{\Der(H)}
  ( X[j-2],T') = 0$ for all $j$ (for $j\neq 2$ since $T$ is a tilting
  complex and for $j=2$ by hypothesis). Therefore $\overline{T}$ is a
  tilting complex if and only if $\Hom_{\Der(H)}
  (F^{-1}X,T'[j])=0$ for all $j \neq 0$, that is, if and only if
  $\Ext^{j}_B(\inj_{X,T},\proj_{T',T})
  \simeq \Hom_{\Der(H)}(\tau X[1],T'[j])\simeq
  \Hom_{\Der(H)}(F^{-1}X,T^{\prime}[j-2])$ equals zero for all $j \neq 2$.
\end{proof}

We can strengthen the former result under an additional hypothesis
on the global dimension of $B$.

\begin{lem}\label{lem:crit-strong}
Let $T=T'\oplus X$ be a tilting complex in $\Der(H)$ 
such that $\HomD(X,T')=0$ and let $B=\End_{\Der(H)}(T)$. 
If $\gldim B\leq 2$, then $\oT=T'\oplus F^{-1}X$ is a tilting complex
in $\Der(H)$ if and only if $\HomD(\tau X,T'[k])=0$ for $k=0,-1$.
\end{lem}

\begin{proof}
  We have $\HomD(F^{-1}X,T'[i])=\HomD(\tau
  X,T'[i+1])=\Ext^{i+2}_B(\inj_{X,T},\proj_{T',T})$, which equals zero
  for all $i\neq 0,-1,-2$.
  
  By Lemma \ref{lem:crit} the complex $T'\oplus F^{-1}X$ is a tilting
  complex if and only if $\HomD(F^{-1} X,T'[i])=0$
  for $i=-1,-2$. Hence the result follows.
\end{proof}

\begin{lem}
\label{lem:nuevo} 
Let $Q$ be a Dynkin quiver and $T$ a tilting complex in
$\Der(H)$. Then $\rho(T) < \tau(\Sigma(T))$.
\end{lem}

\begin{proof}
  As usual, let $T=T'\oplus X$ with $ \rho (T) = T'\oplus F^{-1}X$ and
  $\Sigma = \Sigma (T)$.  Let $\Sigma_1 \in \Sigma$, and let
  $\Sigma_2$ be a maximal element in $\Sigma$ such that $\Sigma_1 \leq
  \Sigma_2$. That is, $\Hom_{\Der(H)}(\Sigma_1, \Sigma_2) \neq 0$, so
  $\Ext^1_{\Der(H)}(\Sigma_2, \tau(\Sigma_1)) \neq 0$ by the Serre
  duality. Our choice of $\Sigma$ implies that $\Sigma_2 \in \add T$,
  so that $\tau \Sigma_1 \notin \add T$ because $T$ is a tilting
  complex in $\Der(H)$. Thus no summand of $T$ is in $\tau \Sigma$ and
  therefore $T' < \tau \Sigma$, since $T' < \Sigma$ by the
  definition of $T'$.

  Since $\add X \subseteq \Sigma$, then $F^{-1}(X) < \tau \Sigma$,
  ending the proof of the lemma.
\end{proof}

\begin{prop}\label{prop:rho-T-tilting}
  Let $T$ be a tilting complex a tilting complex in $\Der(H)$ such
  that $\gldim \End_{\Der(H)}(T)\leq 2$. Then $\rho(T)$ is again a
  tilting complex.
\end{prop}

\begin{proof}
  Again, let $T=T'\oplus X$ and $ \rho (T) = T'\oplus F^{-1}X$. First
  consider the case when $Q$ is a Dynkin quiver and let $\Sigma =
  \Sigma (T)$.  By the lemma we know that $T' < \tau \Sigma $ .  We
  also get $T' < \tau \Sigma [1]$ because $\tau \Sigma < \tau \Sigma
  [1]$.  Since the summands of $X$ are in $\Sigma$, it follows that
  $\Hom_{\Der(H)}(\tau X, T') = 0$ and $\Hom_{\Der(H)}(\tau X, T'[-1]) =
  0$. We conclude from Lemma 3.7 that $\rho ( T)$ is a tilting
  complex.

  Now consider the case where the quiver $Q$ is not Dynkin and let
  $H=\field Q$.
  As in the Definition \ref{def:rolling-nD}, let $m$ be the largest
  half-integer such that $T_{\Rr[m]}\neq 0$. Hence we have $T=T'\oplus
  X$ and $\rho(T)=T'\oplus F^{-1} X$ where $X=T_{\Rr[m]}$. Then clearly
  we have $\HomD(X,T')=0$ and $\HomD(\tau X,T'[k])=0$ for $k=0,-1$
  since $\tau X$ belongs to $\Rr[m]$ and $T'[k]$ to
  $\vee_{i>0}\Rr[m-\frac{i}{2}]$. We conclude again by Lemma
  \ref{lem:crit-strong} that $\rho(T)$ is a tilting complex.
\end{proof}

\begin{remark}\label{rem:gldim}
  The following example shows that the hypothesis on the global
  dimension of the endomorphism algebra is necessary.

  Let $Q=\mathbb{A}_4$ and $T=\bigoplus_{i=1}^4 T_i$ the tilting
  complex in $\Der(H)$ whose relative positions of the indecomposable
  summands $T_i$ are as indicated in the folowing picture.
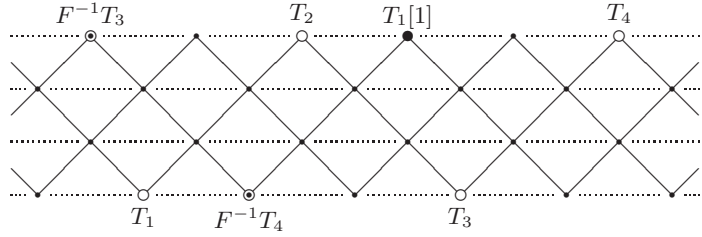
  \begin{center}
    \begin{picture}(240,84)
      \put(0,12){
        \multiput(0,0)(40,0){6}{
          \put(0,0){\line(1,1){40}}
          \put(0,40){\line(1,1){20}}
          \put(0,40){\line(1,-1){40}}
          \put(20,60){\line(1,-1){20}}
          \qbezier[15](5,0)(20,0)(35,0)
          \qbezier[15](5,40)(20,40)(35,40)
        }
        \multiput(0,0)(40,0){5}{
          \qbezier[15](25,20)(40,20)(55,20)
          \qbezier[15](25,60)(40,60)(55,60)
        }
        \multiput(-10,10)(0,40){2}{\line(1,-1){10}}
        \put(-10,30){\line(1,1){10}}
        \multiput(240,0)(0,40){2}{\line(1,1){10}}
        \put(240,40){\line(1,-1){10}}
        \multiput(20,20)(0,40){2}{
          \qbezier[12](-5,0)(-17.5,0)(-30,0)
          }
        \multiput(0,0)(0,40){2}{
          \qbezier[3](-5,0)(-7.5,0)(-10,0)
          }
        \multiput(220,20)(0,40){2}{
          \qbezier[12](5,0)(17.5,0)(30,0)
          }
        \multiput(240,0)(0,40){2}{
          \qbezier[3](5,0)(7.5,0)(10,0)
          }
        \multiput(0,0)(0,40){2}{
          \multiput(0,0)(40,0){7}{\circle*{2}}
          \multiput(20,20)(40,0){6}{\circle*{2}}
        }
        \multiput(40,0)(60,60){2}{\put(0,0){\color{white}\circle*{4}}\put(0,0){\circle{4}}}
        \multiput(160,0)(60,60){2}{\put(0,0){\color{white}\circle*{4}}\put(0,0){\circle{4}}}      
        \multiput(80,0)(-60,60){2}{\put(0,0){\color{white}\circle*{4}}\put(0,0){\circle{4}}\put(0,0){\circle*{2}}}
        \put(40,-5){\HTCenter{\small $T_1$}}
        \put(160,-5){\HTCenter{\small $T_3$}}
        \put(80,-5){\HTCenter{\small $F^{-1}T_4$}}
        \put(100,66){\HBCenter{\small $T_2$}}
        \put(20,65){\HBCenter{\small $F^{-1}T_3$}}
        \put(220,66){\HBCenter{\small $T_4$}}
        \put(140,60){\circle*{4}}
        \put(140,65){\HBCenter{\small $T_1[1]$}}
      }
    \end{picture}
  \end{center}
  Then $B=\End_{\Der(\field Q)}(T)$ has global dimension $3$. By
  Definition \ref{def:rolling-D}, the slice $\Sigma(T)$ is precisely
  the slice containing $T_3$ and $T_4$ and therefore $X=T_3\oplus
  T_4$. Then $\rho(T)$ is not a tilting complex since
  $\Hom_{\Der(\field Q)}(F^{-1} T_4,T_1[1])\neq 0$. 
\end{remark}

\subsection{Global dimension two is preserved}

The next result is fundamental in order for the iteration to work
properly.

\begin{prop}\label{prop:gldim}
  Let $B$ be an iterated tilted algebra. If $\gldim B\leq 2$ then
  $\gldim \rho(B)\leq 2$.
\end{prop}

\begin{proof}
  Let $H$ be a hereditary algebra and $T=T'\oplus X$ a tilting complex
  in $\Der(H)$ such that $B=\EndD(T)$ and $\rho(T)=T'\oplus
  F^{-1}X$. Then we have $\HomD(X,T')=0$ by Remark \ref{rem:hom-X-T'}
  and by Proposition \ref{prop:rho-T-tilting} the complex $\rho(T)$ is
  a tilting complex in $\Der(H)$. To shorten notations we set
  $\oT=\rho(T)$ and $\oB=\rho(B)$.  We shall prove that
  $\Ext^j_{\oB}(\dual\oB,\oB)=0$ for all $j\geq 3$.  Since $\oT$ is a
  tilting complex, we have can show this by proving that
  $\HomD(\tau\oT[1],\oT[j])$ is zero for $j\geq 3$.

  First note that
  \begin{equation}\label{eq:star}
  \HomD(\tau T[1], T[i])=0\quad\text{ for all $i\neq 0,1,2$},
  \end{equation}
  since $\HomD(\tau T[1], T[i])\simeq\Ext^i_B(\dual B,B)$.

  Therefore $\HomD(\tau F^{-1} X[1],F^{-1} X[j])=\HomD(\tau X[1],X[j])=0$ for
  $j\geq 3$ and $\HomD(\tau T'[1],T'[j])=0$ for $j\geq 3$.
  Also, $\HomD(\tau T'[1],F^{-1}X[j])=\HomD(T'[1],X[j-1])$, which is 
  zero for all $j\neq 2$  since $T$ is a tilting complex.  
  
  Hence, it remains to see that $\HomD(\tau^2 X,T'[j])=0$ for $j\geq
  3$.  The minimal projective resolution of $\inj_{X,T}$ in $\mod B$
  $$
  0\ra P_2\ra P_1\ra
  P_0\xrightarrow{\varphi}\inj_{X,T}\ra 0
  $$ gives rise to two exact triangles $\Delta_a\colon K\ra P_0\ra
  \inj_{X,T}\ra K[1]$ and $\Delta_b\colon P_2\ra P_1\ra K\ra
    P_2[1]$, where $K$ denotes the kernel of $\varphi$.

  To both triangles apply first the inverse of the equivalence
  $G\colon\Der(H)\ra\Der(B)$ and then $\tau$, to obtain exact
  triangles of the form $S\ra \tau T_0\ra \tau^2 X[1]\ra S[1]$ and
  $\tau T_2\ra \tau T_1\ra S\ra \tau T_2[1]$
  with $S=\tau G^{-1}(K)$ and some $T_0,T_1,T_2\in\add T$. To these
  triangles apply 
  the homological functor $\HomD(-,T'[j])$ to get exact sequences
  \begin{gather}\label{eq:exseq-a}
    (\tau T_0[1], T'[j])\ra (S[1],T'[j])\ra (\tau^2 X[1],T'[j])\ra (\tau
    T_0,T'[j])\\
    \label{eq:exseq-b}
    (\tau T_2[2],T'[j])\ra (S[1],T'[j])\ra (\tau T_1[1],T'[j]),
  \end{gather}
  where we abbreviated $(Y,Z)=\HomD(Y,Z)$.  By \eqref{eq:star}, the end
  terms of both sequences \eqref{eq:exseq-a} and \eqref{eq:exseq-b}
  are zero for $j>3$ and hence we get $\HomD(\tau^2
  X[1],T'[j])\simeq \HomD(S[1],T'[j])=0$ for $j>3$, which is what we
  wanted to prove.
\end{proof}

\subsection{Iterated rolling}

We now study the iteration of rolling. Fix a quiver $Q$, set $H=\field
Q$. Now start from a given tilting complex $T$ with endomorphism
algebra $B$ with $\gldim B\leq 2$. By Proposition
\ref{prop:rho-T-tilting} the complex $\rho(T)$ is again a tilting
complex and by Proposition \ref{prop:gldim} the endomorphism algebra
$\rho(B)=\EndD(\rho(T))$ satisfies $\gldim \rho(B)\leq 2$. Iterating
we get a sequence of tilting complexes $\rho^h(T)$ with endomorphism
algebras $\rho^h(B)$.  We will show that for sufficiently large
$h$ the algebra $\rho^h(B)$ is tilted.

For this we need some preliminary result in case where $Q$ is Dynkin.
Recall from section \ref{sec:der-struc} that for $Q$ Dynkin, $d(Y,Z)$
denotes the length of the paths in the Auslander-Reiten quiver
$\Gamma$ of $\Der(\field Q)$ from $Y$ to $Z$.

Let $\rho^h(T)=\bigoplus_{i=1}^n T^{(h)}_i$ be the decomposition into
indecomposables and define the natural number
$$
m_h(i)=\sum_{j=1}^n d(T^{(h)}_i,T^{(h)}_j).
$$ 

The following definition will be helpful to simplify the arguments.
\begin{defn}\label{def:h-sigma}
  Let $Q$  be a Dynkin quiver. 
  For each section $\Sigma$ we denote by $H(\Sigma)$ the hereditary
  algebra which has as injectives (concentrated in degree zero) the objects
  in $\Sigma$. That is, we can define $H(\Sigma)[0]=\bigoplus_{i=1 }^n
  \tau^{-1}\Sigma_i[-1]$. Notice that $Q$ and the quiver of $H(\Sigma)$
  coincide up to the orientation of the arrows.
\end{defn}

Now, for each for each $h \geq 0$ 
and each section $\Sigma$ define the set
$$
G_h(\Sigma)=\{i\mid T^{(h)}_i\not\in \mod H(\Sigma)[0]\}
$$
and the natural number
$$
n_h(\Sigma)=\sum_{i\in G_h(\Sigma)} m_h(i).
$$ 
Notice that $n_h(\Sigma)=0$ if and only if $\rho^h(T)\in\mod
H(\Sigma)[0]$.  Finally, let $\Sigma^{(h)}=\Sigma(\rho^h(T))$ be the
section uniquely defined by $\rho^h(T)$ as in section \ref{sec:rolling}. 

\begin{lem}
  \label{lem:dist}
  If $n_h(\Sigma^{(h)})>0$ then
  $n_{h+1}(\Sigma^{(h+1)})<n_h(\Sigma^{(h)})$ and if
  $n_h(\Sigma^{(h)})=0$ then $n_{h+1}(\Sigma^{(h+1)})=0$.
\end{lem}

\begin{proof}
  First suppose that $n_h(\Sigma^{(h)})>0$.  Then, if
  $\rho^h(T)=T'\oplus X$ and $\rho^{h+1}(T)=T'\oplus F^{-1}X$ then for
  $\Sigma'=\tau^2 \Sigma^{(h)}$ we have $F^{-1}X\in \mod
  H(\Sigma')[0]$ and $d(Y,F^{-1} X_i)<d(Y,X_i)$ for all indecomposable
  summands $Y$ of $T'$, $X_i$ of $X$.
  Consequently
  $n_{h+1}(\Sigma')<n_h(\Sigma^{(h)})$ and since clearly
  $n_{h+1}(\Sigma^{(h+1)})\leq n_{h+1}(\Sigma')$ the claim follows.

  If $n_h(\Sigma^{(h)})=0$ then with the same argument as above we
  have $F^{-1}X\in \mod H(\Sigma')[0]$ if
  $\Sigma'=\tau^2\Sigma^{(h)}$. Thus we see $\rho^{h+1}(T)$ belongs to
  $\mod H(\Sigma')[0]$ and consequently $n_{h+1}(\Sigma^{(h+1)})=0$.
\end{proof}

We are now in a position to prove Theorem  \ref{thm:rolling-to-tilted}, stated in the introduction.

\renewcommand{\thedefthm}{\ref{thm:rolling-to-tilted}}
\begin{defthm}
  Let $B$ be an iterated tilted algebra of type $Q$ with $\gldim B\leq
  2$ then for sufficiently large $h$ the algebra $\rho^h(B)$ is tilted
  of type $Q$.
\end{defthm}

\begin{proof}
  Let $H=\field Q$ and $T$ be a tilting complex in $\Der(H)$ such that
  $B=\EndD(T)$.  We have to show that for sufficiently large
  $h$ there exists a hereditary algebra
  $H'$ (which depends on $h$) with $\rho^h(T)\in\mod H'[0]$.

  In case $Q$ is Dynkin this follows directly from Lemma
  \ref{lem:dist}.  In case that $Q$ is not Dynkin we write
  $T=\bigoplus_{i=d}^{s} T_{\Rr[i/2]}$ for some integers $d\leq
  s$.  Then by definition $\rho(T)$ belongs to $\bigcup_{i=d}^{s-1}
  \Rr[i/2]\cup \Rr[\tfrac{s}{2}-1]$. By iterating, we get that for sufficiently
  large $h$ the complex $\rho^h(T)$ belongs to
  $\Rr[p]\cup\Rr[p+\tfrac{1}{2}]$ for some half integer $p$. If $p$ is
  an integer then let $\Sigma_1,\ldots,\Sigma_n$ be any section of
  $\Rr[p+\tfrac{1}{2}]$ such that $T_i\leq \Sigma_j$ for each $j$ and
  each indecomposable summand of $T_{\Rr[p+1/2]}$. Then $H'=H(\Sigma)$
  is a hereditary algebra for which $\rho^h(T)\in\mod H'[0]$.  If $p$
  is a halfinteger then choose a section $\Sigma$ in $\Rr[p]$ such
  that $\Sigma\leq T$ and define $H'$ to be the hereditary
  algebra having its projectives in $\Sigma$. Again we have
  $\rho^h(T)\in\mod H'[0]$.   
\end{proof}

We illustrate the former result by an example.

\begin{exmp}\label{D8}
Let $Q$ be a quiver of type $\mathbb{D}_8$ with some orientation and
$H=\field Q$.  In the following picture the Auslander-Reiten quiver
$\Gamma$ of the derived category $\Der(H)$ is indicated; the arrows
are going from left to right and are drawn as lines to simplify the
picture. The indecomposable summand $T_i$ of the tilting
complex $T=\bigoplus_{i=1}^8 T_i$ has been indicated
by the number $i$ inside a circle, that is, the symbol \begin{picture}(10,8)\put(5,3){\mylabel{$i$}}\end{picture}.
Furthermore, $F^{-1}T_i$, resp. $F^{-2}T_i$, has been indicated by the
symbol 
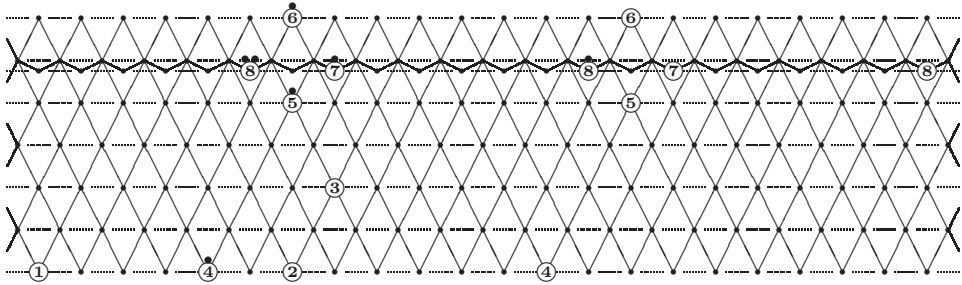
\begin{picture}(10,8)\put(5,3){\mylabeltwo{$i$}{$\bullet$}}\end{picture},
resp.    
\begin{picture}(10,8)\put(5,3){\mylabeltwo{$i$}{$\bullet\bullet$}}\end{picture}.
\begin{center}
  \begin{picture}(352,116)
    \put(0,10){
      \put(0,80){\line(1,2){8}}
      \put(0,48){\line(1,2){24}}
      \put(0,16){\line(1,2){40}}
      \multiput(8,0)(16,0){19}{\line(1,2){48}}
      \put(312,0){\line(1,2){40}}
      \put(328,0){\line(1,2){24}}
      \put(344,0){\line(1,2){8}}
      \put(0,16){\line(1,-2){8}}
      \put(0,48){\line(1,-2){24}}
      \put(0,80){\line(1,-2){40}}
      \multiput(8,96)(16,0){19}{\line(1,-2){48}}
      \put(312,96){\line(1,-2){40}}
      \put(328,96){\line(1,-2){24}}
      \put(344,96){\line(1,-2){8}}
      \multiput(0,80)(16,0){22}{
        \qbezier(0,0)(4,-2)(8,-4)
        \qbezier(8,-4)(12,-2)(16,0)
        }
      \multiput(0,0)(0,32){4}{
        \multiput(-8,0)(16,0){23}{
          \qbezier[7](4,0)(8,0)(12,0)
        }
      }
      \multiput(-8,76)(16,0){23}{
        \qbezier[7](4,0)(8,0)(12,0)
      }
      \multiput(0,16)(0,32){3}{
        \multiput(0,0)(16,0){22}{
          \qbezier[7](4,0)(8,0)(12,0)
        }
      }
      \multiput(0,0)(0,32){3}{
        \qbezier(0,16)(-2,20)(-4,24)
        \qbezier(0,16)(-2,12)(-4,8)
      }
      \multiput(352,0)(0,32){3}{
        \qbezier(0,16)(2,20)(4,24)
        \qbezier(0,16)(2,12)(4,8)
      }
      \multiput(0,0)(0,32){4}{
        \multiput(8,0)(16,0){22}{\mypdott}
      }
      \multiput(8,76)(16,0){22}{\mypdott}
      \multiput(0,0)(0,32){3}{
        \multiput(0,16)(16,0){23}{\mypdott}
      }
      \put(8,0){\mylabel{1}}
      \put(104,0){\mylabel{2}}
      \put(120,32){\mylabel{3}}
      \put(200,0){\mylabel{4}}
      \put(232,64){\mylabel{5}}
      \put(232,96){\mylabel{6}}
      \put(248,76){\mylabel{7}}
      \put(344,76){\mylabel{8}}
      \put(216,76){\mylabeltwo{8}{$\bullet$}}
      \put(88,76){\mylabeltwo{8}{$\bullet\bullet$}}
      \put(72,0){\mylabeltwo{4}{$\bullet$}}
      \put(104,64){\mylabeltwo{5}{$\bullet$}}
      \put(104,96){\mylabeltwo{6}{$\bullet$}}
      \put(120,76){\mylabeltwo{7}{$\bullet$}}
    }
  \end{picture}
\end{center}
We then have 
\begin{align*}
T&= 
T_1\oplus T_2\oplus T_3\oplus T_4 \oplus T_5 \oplus T_6 \oplus T_7
\oplus T_8\\
\rho(T)& =
T_1\oplus T_2\oplus T_3\oplus T_4 \oplus T_5 \oplus T_6 \oplus T_7
\oplus F^{-1}T_8\\
\rho^2(T)& =
T_1\oplus T_2\oplus T_3\oplus F^{-1}T_4 \oplus F^{-1}T_5 \oplus F^{-1}T_6
\oplus F^{-1}T_7 \oplus F^{-1}T_8\\
\rho^3(T)& =
T_1\oplus T_2\oplus T_3\oplus F^{-1}T_4 \oplus F^{-1}T_5 \oplus F^{-1}T_6
\oplus F^{-1}T_7 \oplus F^{-2}T_8
\end{align*}
Define $B_h=\EndD(\rho^h(T))$. The following picture shows $B_h=\field
Q_h/I_h$ for
$h=0,1,2,3$ by a presentation. As usual, relations are indicated by dotted
lines. 
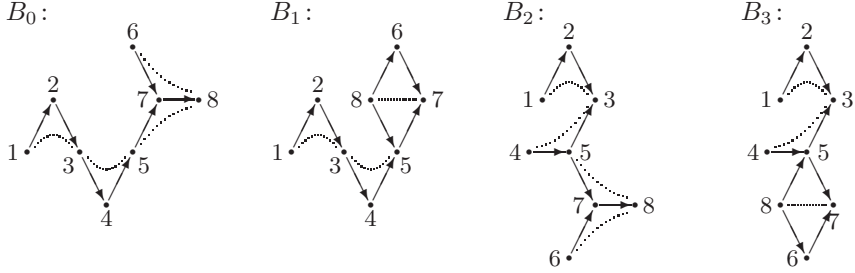
\begin{center}
  \begin{picture}(320,110)
    \put(0,25){
      \put(-8,70){$B_0\colon$}
      \put(0,20){\circle*{2}}
      \put(10,40){\circle*{2}}
      \put(20,20){\circle*{2}}
      \put(30,0){\circle*{2}}
      \put(40,20){\circle*{2}}
      \put(40,60){\circle*{2}}
      \put(50,40){\circle*{2}}
      \put(65,40){\circle*{2}}
      \put(-3,20){\RVCenter{\small 1}}
      \put(10,43){\HBCenter{\small 2}}
      \put(18,18){\RTCenter{\small 3}}
      \put(30,-2){\HTCenter{\small 4}}
      \put(42,18){\LTCenter{\small 5}}
      \put(40,63){\HBCenter{\small 6}}
      \put(46,40){\RVCenter{\small 7}}
      \put(68,40){\LVCenter{\small 8}}
      \put(1,22){\vector(1,2){8}}
      \multiput(11,38)(10,-20){2}{\vector(1,-2){8}}
      \multiput(31,2)(10,20){2}{\vector(1,2){8}}
      \put(41,58){\vector(1,-2){8}}
      \put(51.5,40){\vector(1,0){12}}
      \qbezier[12](3,22)(10,31)(17,22)
      \qbezier[12](23,18)(30,9)(37,18)
      \qbezier[12](43,23)(52,35)(62,37)
      \qbezier[12](43,57)(52,45)(62,43)
    }
    \put(100,25){
      \put(-8,70){$B_1\colon$}
      \put(0,20){\circle*{2}}
      \put(10,40){\circle*{2}}
      \put(20,20){\circle*{2}}
      \put(30,0){\circle*{2}}
      \put(40,20){\circle*{2}}
      \put(40,60){\circle*{2}}
      \put(50,40){\circle*{2}}
      \put(30,40){\circle*{2}}
      \put(-3,20){\RVCenter{\small 1}}
      \put(10,43){\HBCenter{\small 2}}
      \put(19,18){\RTCenter{\small 3}}
      \put(30,-2){\HTCenter{\small 4}}
      \put(41,18){\LTCenter{\small 5}}
      \put(40,63){\HBCenter{\small 6}}
      \put(53,40){\LVCenter{\small 7}}
      \put(27,40){\RVCenter{\small 8}}
      \put(1,22){\vector(1,2){8}}
      \multiput(11,38)(10,-20){2}{\vector(1,-2){8}}
      \multiput(31,2)(10,20){2}{\vector(1,2){8}}
      \multiput(41,58)(-10,-20){2}{\vector(1,-2){8}}
      \put(31,42){\vector(1,2){8}}
      \qbezier[12](3,22)(10,31)(17,22)
      \qbezier[12](23,18)(30,9)(37,18)
      \qbezier[12](33,40)(40,40)(47,40)
    }
    \put(195,5){
      \put(-15,90){$B_2\colon$}
      \put(0,60){\circle*{2}}
      \put(10,80){\circle*{2}}
      \put(20,60){\circle*{2}}
      \put(-5,40){\circle*{2}}
      \put(10,40){\circle*{2}}
      \put(10,0){\circle*{2}}
      \put(20,20){\circle*{2}}
      \put(35,20){\circle*{2}}
      \put(-3,60){\RVCenter{\small 1}}
      \put(10,83){\HBCenter{\small 2}}
      \put(23,60){\LVCenter{\small 3}}
      \put(-8,40){\RVCenter{\small 4}}
      \put(14,40){\LVCenter{\small 5}}
      \put(7,0){\RVCenter{\small 6}}
      \put(16,20){\RVCenter{\small 7}}
      \put(38,20){\LVCenter{\small 8}}
      \put(1,62){\vector(1,2){8}}
      \multiput(11,78)(0,-40){2}{\vector(1,-2){8}}
      \multiput(11,2)(0,40){2}{\vector(1,2){8}}
      \multiput(-3.5,40)(25,-20){2}{\vector(1,0){12}}
      \qbezier[12](3,62)(10,71)(17,62)
      \qbezier[12](13,3)(22,15)(32,17)
      \qbezier[12](13,37)(22,25)(32,23)
      \qbezier[12](-2,42)(8,44)(17,57)
    }    
    \put(285,5){
      \put(-15,90){$B_3\colon$}
      \put(0,60){\circle*{2}}
      \put(10,80){\circle*{2}}
      \put(20,60){\circle*{2}}
      \put(-5,40){\circle*{2}}
      \put(10,40){\circle*{2}}
      \put(10,0){\circle*{2}}
      \put(20,20){\circle*{2}}
      \put(0,20){\circle*{2}}
      \put(-3,60){\RVCenter{\small 1}}
      \put(10,83){\HBCenter{\small 2}}
      \put(23,60){\LVCenter{\small 3}}
      \put(-8,40){\RVCenter{\small 4}}
      \put(14,40){\LVCenter{\small 5}}
      \put(7,0){\RVCenter{\small 6}}
      \put(20,17){\HTCenter{\small 7}}
      \put(-3,20){\RVCenter{\small 8}}
      \put(1,62){\vector(1,2){8}}
      \multiput(11,78)(0,-40){2}{\vector(1,-2){8}}
      \multiput(11,2)(0,40){2}{\vector(1,2){8}}
      \put(-3.5,40){\vector(1,0){12}}
      \put(1,22){\vector(1,2){8}}
      \put(1,18){\vector(1,-2){8}}
      \qbezier[12](3,62)(10,71)(17,62)
      \qbezier[12](3,20)(10,20)(17,20)
      \qbezier[12](-2,42)(8,44)(17,57)
    }    
  \end{picture}
\end{center} 
Note that $B_3$ is tilted. By the above result all algebras $B_i$ for
$i>3$ are also tilted. By calculating the further tilting complexes
$\rho^h(T)$ for $h=4,\ldots,8$ one verifies that $B_h\simeq B_{h+3}$
for $h\geq 5$. 
Observe that in this example all relation extensions
$\ealg(\rho^h(B))$ have isomorphic quivers as shown in the following
picture. 
This is no coincidence and
will be shown in Section \ref{sec:rolling-rel-ext} below.
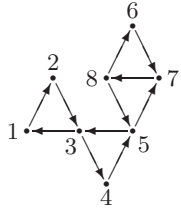
\begin{center}
  \begin{picture}(50,80)
    \put(0,10){
      \put(0,20){\circle*{2}}
      \put(10,40){\circle*{2}}
      \put(20,20){\circle*{2}}
      \put(30,0){\circle*{2}}
      \put(40,20){\circle*{2}}
      \put(40,60){\circle*{2}}
      \put(50,40){\circle*{2}}
      \put(30,40){\circle*{2}}
      \put(-3,20){\RVCenter{\small 1}}
      \put(10,43){\HBCenter{\small 2}}
      \put(19,17){\RTCenter{\small 3}}
      \put(30,-2){\HTCenter{\small 4}}
      \put(42,18){\LTCenter{\small 5}}
      \put(40,63){\HBCenter{\small 6}}
      \put(53,40){\LVCenter{\small 7}}
      \put(27,40){\RVCenter{\small 8}}
      \put(1,22){\vector(1,2){8}}
      \multiput(11,38)(10,-20){2}{\vector(1,-2){8}}
      \multiput(31,2)(10,20){2}{\vector(1,2){8}}
      \multiput(41,58)(-10,-20){2}{\vector(1,-2){8}}
      \put(31,42){\vector(1,2){8}}
      \multiput(18,20)(20,0){2}{\vector(-1,0){16}}
      \put(48,40){\vector(-1,0){16}}
    }    
  \end{picture}
\end{center}
\end{exmp}

The next result shows the importance of iterated tilted algebras
with global dimension less or equal than two. It has been obtained
independently by Osamu Iyama in \cite[Thm. 1.22]{I} and also by Claire
Amiot in \cite[4.10]{A} using different techniques.

\begin{cor}
  Let $H$ be a hereditary algebra.
  If $T$ is a titling complex in $\Der(H)$ such that $\gldim B\leq 2$,
  where $B=\EndD(T)$ then $T$ is a cluster-tilting object in the
  cluster category $\CluCat$ and $C=\EndC(T)$ is a cluster-tilted
  algebra.
\end{cor}

\begin{proof}
  By Theorem \ref{thm:rolling-to-tilted} there exists a number $h$
  such that $\rho^h(B)$ is a tilted algebra. 
  By \cite[Theorem 3.3]{BMRRT}, the object $\rho^h(T)$ defines a
  cluster-tilting object in $\CluCat$ and $C'=\EndC(\rho^h(T))$ is a cluster-tilted algebra. Since $T$ and
  $\rho^h(T)$ define isomorphic objects in $\CluCat$ {the result
  follows.}
\end{proof}

\subsection{Behaviour of the relation extensions under rolling}
\label{sec:rolling-rel-ext}

{Notice that for any} object $T$ of $\Der(H)$, the endomorphism algebra
$$
\EndC(T)=\bigoplus_{i\in \ZZ}\HomD(T,F^{i}T)
$$ is naturally $\ZZ$-graded and contains $B=\EndD(T)$ as a
subalgebra.  Recall from Section \ref{sec:tilt-basic} that if $T$ is a
tilting complex then we have canonically that $\Ext_B^2(\dual
B,B)\simeq\HomD(T,FT)$ with the natural structure of
$B$-$B$-bimodules. Therefore we get a canonical projection
$\pi(B)\colon \EndC(T)\ra \ealg(B)$ of vector spaces and it was proven
  in \cite[Lemma 3.3]{ABS} that $\pi(B)$ is in fact an algebra
  isomorphism when $T$ is a stalk complex concentrated in degree
  zero.  However, in general $\pi(B)$ will not be an algebra
homomorphism. Observe that if $\gldim B\leq 2$ then $\ealg(B)\simeq
\HomD(T,T)\oplus \HomD(T,FT)$.  The next result is straightforward.
 
\begin{lem}\label{lem:seq}
  If $\pi(B)$ is an algebra homomorphism  then 
  the sequence of homomorphisms of algebras
  \begin{equation}\label{eq:seq}
    B\xrightarrow{j} \EndC(T)\xrightarrow{\pi(B)} \ealg(B)\xrightarrow{p} B,
  \end{equation}
  is the identity map, where $j$ and $p$ are the canonical inclusion
  and projection maps respectively. In particular $\EndC(T)$ is a
  split extension of $B$. Moreover, the canonical graded inclusion
  $\delta(B)\colon \ealg(B)\ra C(B)$ is a homomorphism of
  $B$-$B$-bimodules and satisfies
  $\pi(B)\delta(B)=\id_{\ealg(B)}$.\proofend
\end{lem}

The next result shows that the relation extensions are closely related
under rolling.

\begin{prop}\label{prop:ind-ealg}
  Let $T$ be a tilting complex in $\Der(H)$ such that its endomorphism
  algebra $B$ satisfies $\gldim B\leq 2$.  Let $\oT=\rho(T)$ and
  $\oB=\rho(B)$.  Then there exists a canonical algebra homomorphism
  $\Theta\colon \ealg(\oB)\ra \ealg(B)$ which is surjective and whose
  kernel is contained in $\rad^2 \ealg(\oB)$. Furthermore, $\Theta$
  and the canonical isomorphism $\Psi\colon \EndC(\oT)\ra \EndC(T)$
  commute with the projections, that is, $\Theta\pi(\oB) = \pi(B)
  \Psi$. Moreover, if $\pi(\oB)$ is an algebra homomorphism then also
  $\pi(B)$ is an algebra homomorphism.
\end{prop}

\begin{proof}
  The canonical isomorphism $\Psi\colon\EndC(\oT)\ra\EndC(T)$ is given by
  the direct sum of the following bijective maps
  \begin{gather*}
  \id\colon \EndC(T')\ra\EndC(T'),\quad \quad
  \sigma^{-1}\colon \HomC(T',F^{-1} X)\ra \HomC(T',X),\\ 
  \sigma F\colon \HomC(F^{-1} X,T')\ra \HomC(X,T'),\quad\quad
  F\colon \EndC(F^{-1} X)\ra \EndC(X),
  \end{gather*}
  where $\sigma$ denotes the shift in the $\ZZ$-graduation, that is
  $$ 
  \sigma\colon\bigoplus\nolimits_{i\in\ZZ}(Y,F^i Z)\ra
  \bigoplus\nolimits_{i\in\ZZ}(Y,F^{i+1} Z),(f_i)_{i\in\ZZ}\mapsto
  (f_{i+1})_{i\in\ZZ},
  $$
  where we abbreviated again  $(Y,Z)=\HomD(Y,Z)$, as we shall do also
  in the forthcoming.
  Now, $\Theta\colon \ealg(\oB)\ra \ealg(B)$ is defined by the
  following four maps.
  \begin{align}
    \label{eq:it-map-1}
    \id\colon& (T',T')\oplus (T',FT') \ra (T',T')\oplus (T',FT')\\
    \label{eq:it-map-2}
    \left[\begin{matrix}0& \id\\ 0& 0
        \end{matrix}\right]\colon&
    (T',F^{-1}X)\oplus (T',X)\ra (T',X)\oplus (T',FX)\\
    \label{eq:it-map-3}
    \left[\begin{matrix}0& 0\\ F &0
        \end{matrix}\right]\colon&
    (F^{-1}X,T')\oplus(F^{-1}X,FT')\ra (X,T')\oplus (X,FT')\\
    \label{eq:it-map-4}
    F\colon& (F^{-1}X,F^{-1}X)\oplus (F^{-1}X,X)\ra (X,X)\oplus(X,FX).
  \end{align}
  Since by hypothesis $\HomD(T',F X)=0$, resp. $\HomD(X,T')=0$, the
  maps in \eqref{eq:it-map-2}, resp. \eqref{eq:it-map-3} are
  surjective. Therefore the map $\Theta$ is
  surjective and $\Theta \pi(\oB)=\pi(B)\Psi$.

  Now, the kernel of $\Theta$ is clearly $\HomD(T',F^{-1}X)\oplus
  \HomD(F^{-1} X, F T')$, but by Lemma \ref{lem:basic} the first
  summand is zero. We have $\HomD(F^{-1} X, F T')=\HomD(\tau
  (F^{-1}X)[1],T'[2])\simeq
  \Ext_{\oB}^{2}(\inj_{F^{-1}X,\oT},\proj_{T',\oT})$ since $\oT$ is a
  tilting complex. We will show that the last term is contained in the
  radical of $\Ext^2_{\oB}(\dual \oB,\oB)$. By \cite[Section 2.4]{ABS} we
  have $\TOP \Ext^2_{\oB}(\dual \oB,\oB)=\Ext^2_{\oB}(\soc\dual
  \oB,\TOP \oB)$. Hence it suffices to prove that
  $\Ext^2_{\oB}(S_i,S_j)=0$ for all indecomposable simples $S_i$,
  resp. $S_j$, which are direct summands of $\soc \inj_{F^{-1}X,\oT}$,
  resp. $\TOP\proj_{T',\oT}$.

  Suppose the contrary, that is, there exist such summands $S_i$ and
  $S_j$ with $\Ext^2_{\oB}(S_i,S_j)\neq 0$. Let $0\ra Q_2\ra Q_1\ra
  P_i\ra S_i$ be the projective resolution in $\mod \oB$ of
  $S_i$ and $\varphi\colon Q_2\ra S_j$ some morphism defining a non-zero
  element of $\Ext_{\oB}^2(S_i,S_j)$. This shows that some direct summand of
  $Q_2$ is isomorphic to $P_j$ and hence we get a sequence
  $$
  P_j\ra Q'\ra P_i
  $$ 
  of non-zero maps between indecomposable projective $\oB$-modules.
  One of these non-zero morphisms then must map from a summand of
  $\proj_{F^{-1}X,\oT}$ to a summand of $\proj_{T',T}$. This contradicts
  the fact that
  $\Hom_{\oB}(\proj_{F^{-1}X,\oT},\proj_{T',T})=\HomD(X,T')$ equals
  zero.

  It remains to see that if $\pi(\oB)$ is an algebra homomorphism then
  also $\pi(B)$ is an algebra homomorphism. That is, we suppose that
  for all $j\neq 0,1$ and all morphisms 
  \begin{equation}\label{eq:ind-hyp}
  \oT\xrightarrow{f}F^j \oT \xrightarrow{g} \oT\oplus F\oT
  \end{equation}
  the composition $gf$ is zero and have to show that for all
  $h\neq 0,1$ and all morphisms
  $T\xrightarrow{f'}F^h T \xrightarrow{g'} T\oplus F T$
  the composition $g'f'$ is zero. For this we consider $16$
  different combinations:   for $A,B\in\{T',X\}$ and
  $C\in\{T',X,FT',FX\}$, we consider the compositions
  \begin{equation}\label{eq:to-show-2}
    A\xrightarrow{f'}F^h B\xrightarrow{g'} C
  \end{equation}
  for $h\neq 0,1$.  For some of the combinations, the proof that
  $g'f'=0$ is straightfoward using \eqref{eq:ind-hyp}, as for instance
  if $A=B=T'$ and $C=T',X,FT'$.  Also, by hypothesis there is nothing
  to show if $(A,C)$ equals $(X,T')$ or $(T',FX)$. The remaining
  combinations are then divided in two cases:
  \begin{itemize}
  \item[(a)] $A=T'$, $B=X$ and $C\in\{T',X,FT'\}$
  \item[(b)] $A=X$, $B\in\{T',X\}$ and $C\in\{X,FT',FX\}$
  \end{itemize}
  Let $j=h-1$. In case (a) observe that by \eqref{eq:ind-hyp} the
  composition \eqref{eq:to-show-2} holds for all $h\geq 3$ and all
  $h<0$. In case (b), apply $F^{-1}$ to \eqref{eq:to-show-2}, in order
  to see that again the composition is zero if $h\geq 3$ or $h<0$. So
  it only remains to consider the case where $h=2$. In any case $g'=0$
  by Lemma \ref{lem:basic}. This finishes the proof  of the proposition.
\end{proof}

We prove now Theorem \ref{thm:split}, stated in the introduction.  See
also \cite[4.17]{A} for a different proof of the last assertion of the
theorem, relating the quivers of $C$ and $\ealg(B)$.

\renewcommand{\thedefthm}{\ref{thm:split}}
\begin{defthm}
  If $B$ is an iterated tilted algebra of $\gldim B\leq 2$ then there
  exists a cluster-tilted algebra $C$ which is a split extension of
  $B$. More precisely, if $B=\EndD(T)$ with $H$ a hereditary algebra
  and $T$ is a tilting complex in $\Der(H)$ then
  $C=\End_{\CluCat(H)}(T)$ is a cluster-tilted algebra and there
  exists a sequence of algebra homomorphisms
  $$
  B\ra C\xrightarrow{\pi}\ealg(B)\ra B
  $$
  whose composition is the identity map. Moreover, the kernel of $\pi$
  is contained in $\rad^2 C$. In particular $C$ and $\ealg(B)$ have
  the same quivers and are both split extensions of $B$.
\end{defthm}

\begin{proof}
  Let $H$ be a hereditary algebra with quiver $Q$ and 
  $T$ be a tilting complex in $\Der(H)$ such that $B=\EndD(T)$.

  We already know from Theorem \ref{thm:rolling-to-tilted} that for
  sufficiently large $h$ the algebra $\rho^h(B)$ is tilted of type $Q$
  and $C=\EndC(\rho^h(T))$ is cluster-tilted. It follows now from
  \cite[Thm. 3.4]{ABS} that $\pi(\rho^h(B))\colon C\ra \ealg(\rho^h(B))$ is
  an isomorphism.
  Hence by Proposition \ref{prop:ind-ealg} we get inductively for
  $i=h-1,h-2,\ldots,1$ that the projection $\pi(\rho^i(B))$ is an
  algebra homomorphism and $\Theta_i\colon \ealg(\rho^i(B))\ra
  \ealg(\rho^{i-1}(B))$ is surjective with kernel contained in $\rad^2
  \ealg(\rho^i(B))$. Therefore the same holds for the composition
  $\Theta_1\Theta_2\cdots\Theta_h$. Thus $\pi(B)\colon C \rightarrow
    \ealg(B)$ has the same property, because it is obtained from
    $\Theta_1\Theta_2\cdots\Theta_h$ by composing with isomorphisms,
    as follows by repeated application of Proposition
    \ref{prop:ind-ealg} and using that $\pi(\rho^h(B))$ is an
    isomorphism.  In particular $\ealg(B)$ has the same quiver (up to
  isomorphism) as $C$.  By Lemma \ref{lem:seq} the composition of the
  morphisms $B\ra
  \EndC(T)\xrightarrow{\pi(B)} \ealg(B)\ra B$ is the identity map and
  consequently both algebras $C$ and $\ealg(B)$ are split extensions
  of $B$.
\end{proof}

\begin{defn}
  For an iterated tilted algebra $B$ with $\gldim B\leq 2$ choose a
  hereditary algebra $H$ and tilting complex $T$ in $\Der(H)$ with
  $B=\EndD(T)$. We then define $C(B)$ to be the cluster-tilted algebra
  $\EndC(T)$. \end{defn}
We notice that $C(B)\simeq C(\rho(B))$ because $\rho(T) \simeq T$ in
the cluster category $\mathcal C$, so $C(B) \simeq \ealg(\rho^h(B))$
for any $h$ such that $ \rho^h(B)$ is tilted. Such $h$ always exists,
by Theorem \ref{thm:rolling-to-tilted}, and $\rho (B)$ does not depend
on the choices of $H$ and $T$, as observed after Definition
\ref{rolling of B}. It follows that also $C(B)$ is uniquely defined up
to isomorphism independently of the choices of $H$ and $T$.

\begin{prop}
  \label{prop:Ker}
  For each iterated tilted algebra $B$ with $\gldim B\leq 2$ there are
  presentations of the algebras $B$, $\ealg(B)$ and $C(B)$ in which 
  $\{\alpha_1,\ldots,\alpha_r\}$ are the arrows of $B$ and
  $\{\alpha_1,\ldots,\alpha_r,\eta_1,\ldots,\eta_s\}$ are the arrows
  of $\ealg(B)$ and of $C(B)$. Then $\Ker
  \pi(B)=\gen{\eta_1,\ldots,\eta_s}^2$.
\end{prop}

\begin{proof}
  To get the desired presentations one can take a suitable basis of $\rad
  B/\rad^2 B$ and of $\Ext_B^2(\dual B,B)/\rad\Ext_B^2(\dual
  B,B)$. The map $\delta(B)\colon\ealg(B)\ra C(B)$ induces the
  identity on $B$ and satisfies $\delta(B)(\eta_i)=\eta_i$. To
  simplify the notation, we shall
  write $\pi$ and $\delta$ instead of $\pi(B)$ and $\delta(B)$,
  respectively. 

  It
  follows from the definition of the multiplication in $\ealg(B)$ that
  $\gen{\eta_1,\ldots,\eta_s}^2\subseteq \Ker\pi$. By
  Theorem~\ref{thm:split} we have
  $\Ker \pi\subseteq\gen{\alpha_1,\ldots,\alpha_r,\eta_1,\ldots,\eta_s}^2 $.

  Let $z\in \Ker\pi$, say 
  $$
  z=a+\sum_{i=1}^{s} b_i \eta_i c_i+h
  $$
  with $a, b_i$, $c_i\in B$ and $h\in
  \gen{\eta_1,\ldots,\eta_s}^2$. Hence by the above, $h\in\Ker \pi$
  and consequently 
  $$
  y\colon=z-h=a+\sum_{i=1}^s b_i \eta_i c_i
  $$
  belongs to $\Ker \pi$ and also to $\HomD(T,T)\oplus \HomD(T,FT)$, a
  space to which the map $\pi$ restricts as the identity. Hence
  $\pi(y)=0$ implies $y=0$ and consequently
  $z\in\gen{\eta_1,\ldots,\eta_s}{^2}$.
\end{proof}

We observe that though the algebras $C(B)$ and $\ealg(B)$ have the
same quiver, they are in general not isomorphic, not even in the
Dynkin case, as will be shown in Rk.4.20.

\begin{remark}
  Theorem \ref{thm:split} together with \cite[Thm. 2.13]{BMR3} and the
  classification given in \cite{AS2} can be used as
  criteria for discarding an algebra of being iterated tilted, see
  Remark \ref{rem:B-is-cut-of-R(B)} for an example.
\end{remark}

\begin{cor}
  If $B$ is an iterated tilted algebra of Dynkin type with $\gldim
  B\leq 2$ then $\ealg(B)$ is of finite representation type.
\end{cor}

\begin{proof}
  By \cite[Cor. 2.4]{BMR3} the cluster-tilted algebra $C(B)$ is of finite
  representation type. Hence so is $\ealg(B)$ being a quotient of $C(B)$.
\end{proof}

\section{Admissible cuts of cluster-tilted algebras of Dynkin type}
\label{sec:adm-cut}
\subsection{Cluster-tilted algebras of Dynkin type}
\label{sec:rel-Dynkin}
We now want to give a more combinatorial description of the
relationship between an iterated tilted algebra $B$ with $\gldim B\leq
2$, its relation extension $\ealg(B)$ and the corresponding cluster-tilted algebra $C(B)$ in the case where these algebras are of
finite representation type. 

Recall from \cite{BMR} that the quivers of the cluster-tilted algebras
arising from a given cluster category are exactly the quivers
corresponding to the exchange matrices of the associated cluster
algebra. The following result follows therefore from \cite[Thm. 1.8
and Lemma 7.5]{FZ2}

\begin{prop}
  \label{prop:cyc-oriented}
  Each chordless cycle in the quiver $Q_C$ of a cluster-tilted algebra
  $C$ is oriented.
\end{prop}

Also the following result, proven in \cite[Prop. 1.4]{BMR} will
be useful.

\begin{thm}
  \label{thm:idempotent-quotient}
  Let $C$ be a cluster-tilted algebra and $e$ an idempotent of
  $C$. Then $C/CeC$ is again a cluster-tilted algebra.
\end{thm}

\subsection{Relations for cluster-tilted algebras of Dynkin type}

We will need the description of the relations for cluster-tilted
algebras of Dynkin type given in \cite{BMR}. We start by recalling
that if there is an arrow from $i$ to $j$, a path from $j$ to $i$ is
called \emph{shortest} if it contains no proper subpath which is a
cycle and if the full subquiver generated by the path and the arrow
contains no further arrows. A relation $\rho$ is called \emph{minimal}
if whenever $ \rho= \sum_i \beta_i \rho_i\gamma_i $ where $\rho_i$ is
a relation for every $i$, then $\beta_i $ and $\gamma_i$ are scalars
for some index $i$ (see \cite{BMR}).

The following definition will simplify the language.

\begin{defn}[Parallel and antiparallel paths] 
An arrow $\alpha$ is called \emph{parallel},
(resp. \emph{antiparallel}) to a relation (or a path or an arrow) $\rho$ if
$s(\alpha)=s(\rho)$ and $t(\alpha)=t(\rho)$ (resp. $s(\alpha)=t(\rho)$
and $t(\alpha)=s(\rho)$).
\end{defn}

The following description is an immediate consecuence of
\cite[Thm. 4.1]{BMR}.

\begin{thm}
\label{thm:Dynkin-rel} 

Let $C=\field Q_C/I_C$ be a cluster-tilted algebra of Dynkin
type. Then, in $Q_C$ for each arrow $\eta$ there exist at most two
shortest antiparallel paths to $\eta$. If there is at least one and
$\Sigma_\eta$ denotes the full subquiver of $Q_C$ given by the
vertices of $\eta$ and the antiparallel paths, then the quiver
$\Sigma_\eta$ is isomorphic to $C(n)$ (for some $n$) or to $G(a,b)$
(for some $a,\,b$), as shown in the following picture.
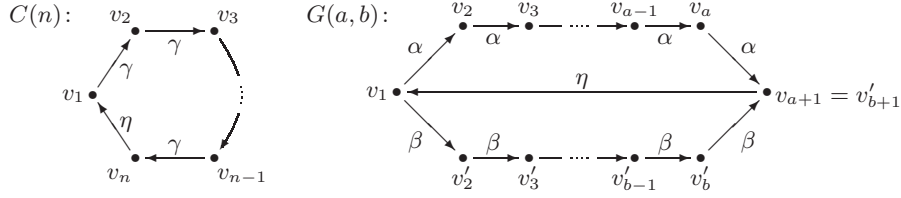
\begin{center}
  \begin{picture}(320,70)
    \put(20,10){
      \put(-5,55){\RVCenter{\small $C(n)\colon$}}
      \put(0,24){\circle*{3}}
      \multiput(16,0)(30,0){2}{\circle*{3}}
      \multiput(16,48)(30,0){2}{\circle*{3}}
      \put(2,27){\vector(2,3){12}}
      \put(20,48){\vector(1,0){22}}
      \put(42,0){\vector(-1,0){22}}
      \put(14,3){\vector(-2,3){12}}
      \qbezier(48,45)(53,37.5)(55,31.5)
      \put(48,3){\vector(-2,-3){0.001}}
      \qbezier(49,4.5)(53,10.5)(55,16.5)
      \qbezier[3](56,27)(56.5,24)(56,21)
      \put(-3,24){\RVCenter{\small $v_1$}}
      \put(15,52){\RBCenter{\small $v_2$}}
      \put(47,52){\small $v_3$}
      \put(47,-4){\LTCenter{\small $v_{n-1}$}}
      \put(15,-4){\RTCenter{\small $v_n$}}
      \put(10,12){\small $\eta$}
      \put(10,36){\LTCenter{\small$\gamma$}}
      \put(31,3){\HBCenter{\small $\gamma$}}
      \put(31,45){\HTCenter{\small $\gamma$}}
    }
    \put(135,10){
      \put(0,55){\RVCenter{\small $G(a,b)\colon$}}
      \multiput(0,0)(0,50){2}{
        \multiput(25,0)(25,0){2}{\circle*{3}}
        \multiput(90,0)(25,0){2}{\circle*{3}}
        \multiput(29,0)(65,0){2}{\vector(1,0){17}}
        \put(54,0){\line(1,0){8}}
        \put(76,0){\vector(1,0){11}}
        \qbezier[3](66,0)(69,0)(72,0)
      }
      \multiput(0,25)(140,0){2}{\circle*{3}}
      \multiput(3,22)(115,25){2}{\vector(1,-1){19}}
      \multiput(3,28)(115,-25){2}{\vector(1,1){19}}
      \put(136,25){\vector(-1,0){132}}
      \put(-3,25){\RVCenter{\small $v_1$}}
      \put(25,-3){\HTCenter{\small $v_2'$}}
      \put(50,-3){\HTCenter{\small $v_3'$}}
      \put(90,-3){\HTCenter{\small $v_{b-1}'$}}
      \put(115,-3){\HTCenter{\small $v_b'$}}      
      \put(25,55){\HBCenter{\small $v_2$}}
      \put(50,55){\HBCenter{\small $v_3$}}
      \put(90,55){\HBCenter{\small $v_{a-1}$}}
      \put(115,55){\HBCenter{\small $v_a$}}      
      \put(143,25){\LVCenter{\small $v_{a+1}=v'_{b+1}$}}
      \put(70,28){\HBCenter{\small $\eta$}}
      \multiput(36.5,3)(65,0){2}{\HBCenter{\small $\beta$}}
      \multiput(36.5,47)(65,0){2}{\HTCenter{\small $\alpha$}}
      \put(10,10){\RTCenter{\small $\beta$}}
      \put(10,40){\RBCenter{\small $\alpha$}}
      \put(130,10){\LTCenter{\small $\beta$}}
      \put(130,40){\small $\alpha$}
    }
  \end{picture}
\end{center}

The ideal $I_C$ is generated by minimal zero relations and minimal
commutativity relations, and each of them is antiparallel to exactly
one arrow. If an arrow $\eta$ is antiparallel to the minimal zero
relation $\rho$, then $\Sigma_\eta \simeq C(n)$ and $\rho =
\gamma^{n-1} $. If $\eta$ is antiparallel to the minimal commutativity
relation $\rho_1 = \rho_2$, then $\Sigma_\eta \simeq G(a,b)$ and
$\rho_1 = \alpha^a \neq 0, \rho_2 = \beta^b \neq 0$.
\end{thm}

Hence each arrow in an oriented cycle is antiparallel to precisely one
minimal relation (up to scalars), and the relations obtained this way
form a minimal set of generators of $I_C$.

\begin{lem}\label{lem:no-bypass}
  Let $C$ be a cluster-tilted algebra of Dynkin type with quiver
  $Q$. Then for each arrow $\alpha$ there is no other shortest path
  than $\alpha$ which is parallel to $\alpha$ in $Q$.
\end{lem}

\begin{proof}
  Assume otherwise, that is, there exists a path $\gamma$ parallel to
  $\alpha$ which is different in $Q$. Since $C$ is of finite
  representation type, $\gamma$ can not be an arrow. Let
  $\gamma=\gamma_t\gamma_{t-1}\cdots\gamma_1$ be as follows.
  $$
  x_0\xrightarrow{\gamma{1}}x_1\ra\cdots\ra
  x_{t-2}\xrightarrow{\gamma_{t-1}} x_{t-1} \xrightarrow{\gamma_t}x_t
  $$
  By Proposition \ref{prop:cyc-oriented}, 
  the cycle $\alpha\gamma$ is not
  chordless. Let $m\geq 0$ be minimal such that there exists an arrow
  between $x_m$ and $x_s$ for some $s>m+1$. Then let $M$ with
  $m+1<M\leq t$ be maximal such that there exists an arrow $\delta$ between
  $x_m$ and $x_M$. Then the arrows
  $$
  \alpha,\gamma_t,\ldots,\gamma_{M+1},\delta,\gamma_{m},\ldots,\gamma_1
  $$
  form a non-oriented cycle which by contruction is chordless, in
  contradiction to Proposition \ref{prop:cyc-oriented}.
\end{proof}

\subsection{Definition of Admissible cut}
We are now ready to give the combinatorial description of how the
iterated tilted algebras $B$ with $\gldim B\leq 2$ can be obtained
from a cluster-tilted algebra $C$. For this we introduce the following
concept. 

\begin{defn}[Admissible cut]
A subset of the set of arrows $Q_1$ of a quiver $Q$ is called
\emph{admissible cut} of $Q$ if it contains exactly one arrow of each
oriented chordless cycle in $Q$. 
\end{defn}

\begin{remark}
  Let $\Delta$ be an admissible cut of a quiver $Q$. It is
  straightforward to check that for 
  $\alpha\in\Delta$ each arrow $\beta$ of $Q$ which is parallel to
  $\alpha$ also belongs to $\Delta$. 
\end{remark}

\begin{defn}[Quotient by an admissible cut]
  Let $C=\field Q_C/I$ be an algebra given by a quiver $Q_C$ and an
  admissible ideal $I$. A \emph{quotient of $C$ by an admissible cut}
  (or \emph{an admissible cut of $C$)} is an algebra of the form
  $\field Q_C/\gen{I\cup \Delta}$ where $\Delta$ is an admissible cut
  of $Q_C$.
\end{defn}

This is, $B$ is an admissible cut of $C$ if $B$ is the algebra
obtained by deleting in $Q_C$ the arrows of an admissible cut $\Delta$
and considering the induced relations.

\begin{remark}
\label{rem:cut-presentation}
  The definition is not independent of the presentation
  of $B$, that is, for two ideals $I_1$ and $I_2$ such that $\field
  Q/I_1\simeq \field Q/I_2$ the same cut may give
  non-isomorphic quotients $\field Q/\gen{I_1\cup \Delta}\not\simeq
  \field Q/\gen{I_2\cup \Delta}$, as shows the following example. Let
  $Q$ be the quiver as given in the following picture.
  
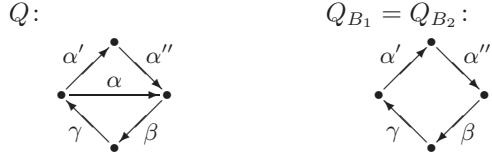
\begin{center}
    \begin{picture}(160,62)
      \put(0,4){
        \put(-20,48){$Q\colon$}
        \multiput(0,20)(40,0){2}{\circle*{3}}
        \multiput(20,0)(0,40){2}{\circle*{3}}
        \put(2,22){\vector(1,1){16}}
        \put(22,38){\vector(1,-1){16}}
        \put(38,18){\vector(-1,-1){16}}
        \put(18,2){\vector(-1,1){16}}
        \put(3,20){\vector(1,0){34}}
        \put(20,23){\HBCenter{\small $\alpha$}}
        \put(9,32){\RBCenter{\small $\alpha'$}}
        \put(32,32){\small $\alpha''$}
        \put(8,8){\RTCenter{\small $\gamma$}}
        \put(31,9){\LTCenter{\small $\beta$}}
      }
      \put(120,4){
        \put(-20,48){$Q_{B_1}=Q_{B_2}\colon$}
        \multiput(0,20)(40,0){2}{\circle*{3}}
        \multiput(20,0)(0,40){2}{\circle*{3}}
        \put(2,22){\vector(1,1){16}}
        \put(22,38){\vector(1,-1){16}}
        \put(38,18){\vector(-1,-1){16}}
        \put(18,2){\vector(-1,1){16}}
        \put(9,32){\RBCenter{\small $\alpha'$}}
        \put(32,32){\small $\alpha''$}
        \put(8,8){\RTCenter{\small $\gamma$}}
        \put(31,9){\LTCenter{\small $\beta$}}
      }
    \end{picture}
  \end{center}
  Furthermore, let $I_1=\gen{\beta\alpha,\ \gamma\beta }$ and
  $I_2=\gen{\beta(\alpha-\alpha''\alpha'),\, \gamma\beta}$. Then the
  quotients $\field Q/I_1$ and $\field Q/I_2$ are
  isomorphic. Furthermore $\Delta=\{\alpha\}$ is an admissible cut but
  the quotients $B_1=\field Q/\gen{I_1\cup\Delta}$ and $B_2=\field
  Q/\gen{I_2\cup \Delta}$ are non-isomorphic since
  $\gen{I_1\cup\Delta}=\gen{\alpha,\gamma\beta}$ whereas
  $I_2=\gen{\alpha,\beta\alpha''\alpha',\gamma\beta}$, that is, $B_2$
  is a proper quotient of $B_1$. 
\end{remark}

However, an admissible cut of a cluster-tilted algebra $C$ of Dynkin
type is independent of the presentation of $C$. This follows from the
next lemma, and the fact that any such algebra $C$ is schurian, that
is, $\dim_k\ e_yCe_x \leq 1$ for any pair of vertices $x, y \in
Q_C$. See \cite [Lemma 1.8] {BMR}.

\begin{lem}
  If $C$ is a schurian algebra and $\Delta$ an
  admissible cut of the quiver $Q$ of $C$ then the quotient of $C$ by
  $\Delta$ is independent of the presentation of $C$.
\end{lem}

\begin{proof}
  Let $f\colon kQ/I \rightarrow kQ/J$ be an isomorphism. By composing, if
  necessary, with the isomorphism of $kQ$ induced by an isomorphism of
  the quiver $Q$, we may assume that $f(e_x) = e_x$, for each $x \in
  Q_0$.

  Since $C$ is schurian, $\dim_k \ e_y (kQ/J)e_x \leq 1$ for each $x, y
  \in Q_0$.  So for each arrow $\alpha$ we have that $f(\alpha) =
  \lambda_\alpha \alpha$ for some non-zero $\lambda_\alpha \in k$.
  Thus if $ \Delta$ is an admissible cut of $ Q $ then $\Delta $ and
  $f(\Delta)$ generate the same ideal in $kQ/J$, and therefore the map
  $ KQ/(I\cup \Delta)\rightarrow kQ/(J\cup \Delta)$ induced by $f$ is
  an isomorphism.
\end{proof}

Notice that the example given in Remark \ref{rem:cut-presentation}
also shows that it is possible that the quiver $Q_{B_1}$ of a quotient
of an algebra $C$ by an admisible cut may have oriented chordless
cycles. However, this can not happen in case where $C$ is a
cluster-tilted algebra of Dynkin type.

\begin{lem}\label{lem:no-oriented}
  Let $C$ be a cluster-tilted algebra of Dynkin type  and $\Delta$
  an admissible cut of the quiver $Q_C$ of $C$. Then for any
  presentation $C=\field Q_C/I$, the quiver $Q_B$ of the quotient $B=\field
  Q_C/\gen{I\cup\Delta}$ has no oriented chordless cycle. 
\end{lem}

\begin{proof}
  Suppose the contrary, namely that in $Q_B$ there exists an oriented
  chordless cycle, given by a path
  $$
  \gamma\colon x_0\xrightarrow{\gamma_1}x_1\ra\cdots\ra
  x_{t-2}\xrightarrow{\gamma_{t-1}}
  x_{t-1}\xrightarrow{\gamma_t}x_t=x_0
  $$
  Then $\gamma$ cannot be chordless in $Q_C$ by
  the definition of admissible cut. Thus there exists an arrow
  $\alpha$ between $x_r$ and $x_s$ for some $s>r+1$. After renumbering
  the vertices $x_i$ and the arrows $\gamma_i$ we can assume without
  loss of generality that $\alpha\colon x_0\ra x_s$ for some $s$ with
  $1<s<t$. This contradicts Lemma \ref{lem:no-bypass}.
\end{proof}

\subsection{Existence of admissible cuts}

We start by the observation that there exist quivers which do not
admit an admissible cut.

\begin{exmp}
  Let $Q$ be the following quiver.
  
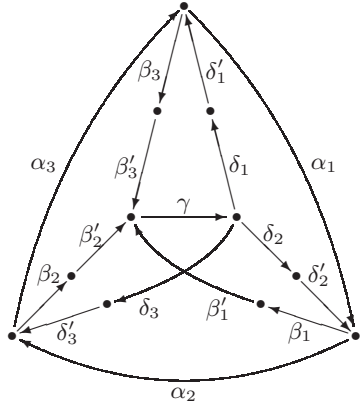
\begin{center}
    \begin{picture}(200,150)
      \put(100,70){
        \put(0,80){\circle*{3}}
        \multiput(-65,-45)(130,0){2}{\circle*{3}}
        \multiput(-20,0)(40,0){2}{\circle*{3}}
        \multiput(-10,40)(20,0){2}{\circle*{3}}
        \put(42.5,-22.5){\circle*{3}}
        \put(-42.5,-22.5){\circle*{3}}
        \put(29,-33){\circle*{3}}
        \put(-29,-33){\circle*{3}}
        \qbezier(2,78)(50,30)(64.1,-41.3)
        \put(64.1,-41.3){\vector(1,-3){0.01}}
        \qbezier(-2,78)(-50,30)(-64.1,-41.3)
        \put(-2,78){\vector(1,1){0.001}}
        \qbezier(-62,-46.5)(0,-77.5)(62,-46.5)
        \put(-62,-46.5){\vector(-2,1){0.001}}
        \put(-16,0){\vector(1,0){32}}
        \multiput(19.2,3.2)(-10,40){2}{\vector(-1,4){8.4}}
        \multiput(-0.8,76.8)(-10,-40){2}{\vector(-1,-4){8.4}}
        \multiput(22,-2)(22.5,-22.5){2}{\vector(1,-1){18.5}}
        \multiput(-63,-43)(22.5,22.5){2}{\vector(1,1){18.5}}
        \put(62,-44){\vector(-3,1){30}}
        \put(-32,-34){\vector(-3,-1){30}}
        \qbezier(26,-32)(-10,-20)(-18.5,-3)
        \put(-18.5,-3){\vector(-1,2){0.001}}
        \qbezier(-26,-32)(10,-20)(18.5,-3)
        \put(-26,-32){\vector(-3,-1){0.001}}
        \put(53,20){\HVCenter{\small $\alpha_1$}}
        \put(-53,20){\HVCenter{\small $\alpha_3$}}
        \put(0,-66.7){\HVCenter{\small $\alpha_2$}}
        \put(0,3){\HBCenter{\small $\gamma$}}
        \put(8.2,56){\LVCenter{\small $\delta'_1$}}
        \put(-8.2,56){\RVCenter{\small $\beta_3$}}
        \put(17,20){\LVCenter{\small $\delta_1$}}
        \put(-17,20){\RVCenter{\small $\beta'_3$}}
        \put(51,-21){\HVCenter{\small $\delta'_2$}}
        \put(34,-6){\HVCenter{\small $\delta_2$}}
        \put(44,-43){\HVCenter{\small $\beta_1$}}
        \put(13,-34){\HVCenter{\small $\beta'_1$}}
        \put(-51,-21){\HVCenter{\small $\beta_2$}}
        \put(-35,-6){\HVCenter{\small $\beta'_2$}}
        \put(-44,-43){\HVCenter{\small $\delta'_3$}}
        \put(-13,-34){\HVCenter{\small $\delta_3$}}
      }
    \end{picture}
  \end{center}
  The only chordless cycles in $Q$ are given by the paths 
  $$
  \alpha_3\alpha_2\alpha_1,\quad\quad
  \delta'_i\delta_i\gamma\beta'_i\beta_i\alpha_i,\quad\quad
  \delta'_{i+1}\delta_{i+1}\gamma\beta'_i\beta_i
  $$ 
  for $i=1,2,3$ where the indices have to be taken modulo $3$.
  
  Suppose that there exists an admissible cut $\Delta$ in $Q$. Then
  one (and only one) of the arrows $\alpha_i$ has to belong to
  $\Delta$. Because of the cyclic symmetry (by interchanging the
  indices cyclically modulo $3$) we can without loss of generality
  assume that $\alpha_1$ belongs to $\Delta$. Since 
 $\delta_1'\delta_1\gamma\beta_1'\beta_1\alpha_1$
 is a chordless cycle we have 
 $\delta_1',\delta_1,\gamma,\beta_1',\beta_1\not \in \Delta$. Since
  $\delta'_1\delta_1\gamma\beta'_3\beta_3$
  (resp. $\delta'_2\delta_2\gamma\beta'_1\beta_1$) is a chordless
  cycle, one (and only one) of the arrows $\beta_3$ or $\beta'_3$
  (resp. $\delta_2$ or $\delta'_2$) must also belong to $\Delta$. We
  can assume the two arrows are $\beta_3$ and $\delta_2$ since the
  argument for any other choice is completely similar.

  Let $C$ be the set of chordless cycles which contain an arrow from 
  $\Delta'=\{\alpha_1,\beta_3,\delta_2\}$. Observe that $C$ contains
  all chordless cycles except $\delta'_3\delta_3\gamma\beta'_2\beta_2$
  and that each arrow of $Q$ occurs in one of the cycles in $C$. Hence
  on one hand the admissible cut $\Delta$ must contain another arrow
  from $\delta'_3\delta_3\gamma\beta'_2\beta_2$ and on the other hand
  $\Delta$ can not contain any more since otherwise one of the cycles
  of $C$ would contain two arrows from$\Delta$, a contradiction. This
  proves that $Q$ does not admit an admissible cut.
\end{exmp}

The following result shows that the quiver of any cluster-tilted
algebra of Dynkin type admits an admissible cut. 

\begin{prop}
\label{prop:cut-exists}
Let $B$ be an iterated tilted algebra of Dynkin type  with
$\gldim B\leq 2$. Then $B$ is an admissible cut of the corresponding
cluster-tilted algebra $C(B)$.
\end{prop}

\begin{proof}
  Suppose that $B$ is not an admissible cut of $C=C(B)$. Then there
  exists a chordless cycle $\gamma$ in the quiver $Q_C$ of $C$ which
  contains at least two arrows which do not belong to the quiver $Q_B$
  of $B$. Denote by $\gamma_L\gamma_{L-1}\ldots\gamma_1$ the path
  obtained by passing along the cycle starting from some vertex
  $s(\gamma_1)$ of $\gamma$ and let $\Phi$ be the set of vertices such
  that $\{\gamma_j\mid j\in \Phi\}$ are the arrows which do not
  belong to $Q_B$. 

  Write $B=\field Q_B/I_B$ and $C=\field Q_C/I_C$.
  Now, by Theorem \ref{thm:split} we have $B=C/J$ for some ideal $J$ of $C$
  with $J\subseteq \rad^2 C$ and the arrows of $Q_C$ coincide with the
  arrows of $Q_{\ealg (B)}$. For each $j\in\Phi$ the arrow
  $\gamma_j$ corresponds to a generating
  relation $\rho_j$ since $\ealg(B)$ is the relation
  extension of $B$. 

  Observe that
  $\delta=\gamma_{j-1}\gamma_{j-2}\ldots\gamma_1\gamma_L\ldots
  \gamma_{j+1}$ is one a path in $Q_C$ which is antiparallel to
  $\gamma_j$ and $\delta$ is not contained in the path algebra $Q_B$
  since by hypothesis $\Phi$ consists of at least two elements.  By
  Theorem \ref{thm:Dynkin-rel},
  there are at most two paths in $Q_C$ which are antiparallel to
  $\gamma_j$ and therefore there exists precisely one path $\delta'$
  in $Q_B$ which is antiparallel to $\gamma_j$ in $Q_C$.  Consequently
  $\rho_j=\delta'$ is a zero relation. Hence the smallest full
  subquiver of $Q_C$ containing $\delta$ and $\delta'$ is isomorphic
  to $G(a,b)$, defined as in Section \ref{sec:rel-Dynkin}. Since
  $C(B)$ is a split extension of $B$, by \cite[2.3]{ACT} it follows
  that the ideal $I_B$ is contained in $I_C$. Thus $\delta'=0$ in $C$
  in contradiction to Theorem
  \ref{thm:Dynkin-rel}.
\end{proof}

\begin{remark}\label{rem:B-is-cut-of-R(B)}
By Theorem \ref{thm:split}, for $B$ an iterated tilted algebra with
$\gldim B\leq 2$, the algebras $C(B)$ and $\ealg(B)$ have the same
quiver, and therefore if $B$ is of Dynkin type $Q$ then $B$ is the
quotient of $\ealg(B)$ by an admissible cut. This however is not true
in general, as shows the following example. Let $B=\field Q_B/I_B$ be the 
algebra presented on the left hand side in the following picture. Then
the quiver of $\ealg(B)$ is as depicted on the right hand.
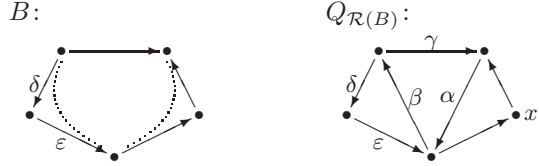
\begin{center}
  \begin{picture}(120,60)
    \put(0,0){
      \put(-40,52){$B\colon$}
      \multiput(-20,40)(40,0){2}{\circle*{3}}
      \multiput(-32,16)(64,0){2}{\circle*{3}}
      \put(0,0){\circle*{3}}
      \put(-17,40){\vector(1,0){34}}
      \put(-21.5,37){\vector(-1,-2){9}}
      \put(-29,14.5){\vector(2,-1){26}}
      \put(3,1.5){\vector(2,1){26}}
      \put(30.5,19){\vector(-1,2){9}}
      \qbezier[20](-20,36)(-30,17)(-5,4)
      \qbezier[20](20,36)(30,17)(5,4)
      \put(-28,28){\RVCenter{\small $\delta$}}
      \put(-18,5){\RVCenter{\small $\varepsilon$}}
    }
    \put(120,0){
      \put(-40,52){$Q_{\ealg(B)}\colon$}
      \multiput(-20,40)(40,0){2}{\circle*{3}}
      \multiput(-32,16)(64,0){2}{\circle*{3}}
      \put(0,0){\circle*{3}}
      \put(-17,40){\vector(1,0){34}}
      \put(-21.5,37){\vector(-1,-2){9}}
      \put(-29,14.5){\vector(2,-1){26}}
      \put(3,1.5){\vector(2,1){26}}
      \put(30.5,19){\vector(-1,2){9}}
      \put(18.5,37){\vector(-1,-2){17}}
      \put(-1.5,3){\vector(-1,2){17}}
      \put(0,42){\HBCenter{\small $\gamma$}}
      \put(-9,23){\LVCenter{\small $\beta$}}
      \put(9,23){\RVCenter{\small $\alpha$}}
      \put(-28,28){\RVCenter{\small $\delta$}}
      \put(-18,5){\RVCenter{\small $\varepsilon$}}
      \put(35,16){\LVCenter{\small $x$}}
    }
  \end{picture}
\end{center}
Clearly $\Delta=\{\beta,\alpha\}$ is not an admissible cut of
$Q_{\ealg(B)}$ since the cycle given by the path $\gamma\beta\alpha$
contains two arrows from $\Delta$. However $B$ is not an iterated
tilted algebra as shows the following argument. Suppose that $B$ is
iterated tilted of type $Q$. Then by Theorem \ref{thm:split} the
algebras $C(B)$ and $\ealg(B)$ have the same quiver and both are split
extensions of $B$. In particular, since $\varepsilon\delta=0$ in $B$
we have also $\varepsilon\delta=0$ in $C(B)$. For the ideal
$J=C(B)\trivpath{x}C(B)$, the quotient $C'=C(B)/J$ is again a
cluster-tilted algebra by Theorem \ref{thm:idempotent-quotient}. 
By \cite[Thm. 2.3]{BORS} there is a unique cluster-tilted algebra with
quiver $Q_{C'}$ and that algebra is known to be of Dynkin type
$\mathbb{D}_4$. This contradicts the description of relations in
\cite{BMR2}, see Section \ref{sec:rel-Dynkin}, where
$\varepsilon\delta\neq 0$. This shows that $B$ is not iterated tilted.
\end{remark}

\begin{remark}
\label{rem:B-is-cut-of-R(B)-bis}
  (a) Each iterated tilted algebra $B$ with $\gldim B\leq 2$ is the
  quotient of $\ealg(B)$ by the ideal $\Delta = \Ext^2_B(DB,B)$,
  which is generated by arrows corresponding to relations of $B$. It
  is unknown to the authors whether each such algebra $B$ is an
  admissible cut of $\ealg(B)$ by $\Delta$.

  (b) We observe that $B$ is an admissible cut of $\ealg(B)$ by
  $\Delta = {\rm Ext}^2_B(DB,B)$ if and only if $B$ is an admissible
  cut of $C(B)$ by $\Delta$.  To prove this statement, assume that $B$
  is the quotient of the relation-extension $\ealg(B)$ by the
  admissible cut $\Delta$.  By Theorem \ref{thm:split} the
  algebras $\ealg(B)$ and $C=C(B)$ are split extensions of $B$ and
  have isomorphic quivers. Therefore $\Delta$ is also an admissible
  cut of the quiver $Q_C$ of $C=\field Q_C/I_C$ and the arrows of $B$
  can be identified with the arrows of $C$ which are not in
  $\Delta$. Let $J$ be the ideal of $C$ such that $B\simeq C/J$.  By
  the above we have $J\supseteq \gen{I_C\cup \Delta}$ and it remains
  to show that $J\subseteq \gen{I_C\cup \Delta}$.  So let $\rho$ be a
  relation of $\field Q_C$ which belongs to $J$. Let
  $\rho=\sum_{i=1}^t\lambda_i\rho_i$ for some non-zero scalars
  $\lambda_i$ and some parallel paths
  $\rho_i=\rho_{i,N_i}\rho_{i,N_i-1}\cdots \rho_{i,1}$. If some
  $\rho_{i,j}\in \Delta$ then $\rho'=\rho-\lambda_i\rho_i\in J$ and by
  induction over the number of summands we can assume that
  $\rho'\in\gen{I_C\cup \Delta}$. Hence it remains to consider the
  case where no summand of $\rho$ contains an arrow of $\Delta$, that
  is, $\rho$ can be considered as element of $\field Q_B$.  Let $\pi\colon
  C \rightarrow \ealg(B)$ be the surjective algebra morphism of
  Theorem 1.1 and $\mu\colon \ealg(B) \rightarrow B$ the canonical
  map. Then $\mu \pi|_B = id_B$ and $\orho=\pi(\orho)$, where $\orho$
  denotes both the class of $\rho$ in the quotient $\field Q_B/I_B$
  and in $\field Q_C/I_C$.  Therefore
  $0=\mu(\orho)=\mu\pi(\orho)=\orho$ shows that indeed $\rho \in I_C$.

  (c) It is interesting to notice that the fact that both $\ealg(B)$
  and $C$ are split extensions of $B$ is essential for the preceding
  statement to hold. Let $C$, $D$ be algebras such that $D$ is a
  quotient of $C$ inducing an isomorphism of quivers $Q_D =
  Q_C$. Clearly the sets of arrows which are admissible cuts for the
  quivers of the two algebras are the same. However, if an algebra $B$
  is an admissible cut of $D$, then it is not always true that $B$ is
  also an admissible cut of $C$, as the following simple example
  shows.

  Let $Q$ be the quiver
  
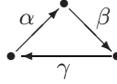
\begin{center}
    \begin{picture}(40,32)
      \put(0,-10){
        \multiput(0,20)(40,0){2}{\circle*{3}}
        \put(20,40){\circle*{3}}
        \put(2,22){\vector(1,1){16}}
        \put(22,38){\vector(1,-1){16}}
        \put(38,20){\vector(-1,0){34}}
        \put(20,17){\HTCenter{\small $\gamma$}}
        \put(9,32){\RBCenter{\small $\alpha$}}
        \put(32,32){\small $\beta$}
      }  
    \end{picture}
  \end{center}

  $C = kQ/\gen{\gamma \beta \alpha}$ and $D = C/ \gen{\beta \alpha}.$ Then $B= D/\gen{\gamma}
  \simeq C / \gen {\gamma, \beta \alpha}$ is an admissible cut of $D$,
  but is not an admissible cut of $C$. Observe that $C$ is not a split
  extension of $B$ since $I_B\neq 0$.
\end{remark}

\subsection{Admissible cuts and antiparallel relations}

We now start the investigation on quotients of cluster-tilted algebras
by admissible cuts by the following basic fact.

\begin{prop}
\label{prop:fund}
Let $B$ be a quotient by an admissible cut of a cluster-tilted algebra
$C$ of Dynkin type. Write $B=\field Q_B/I_B$ where $Q_B$ is the quiver
of $B$ and $I_B$ is an admissible ideal generated by the minimal set
of minimal relations $\{\rho_i\mid i=1,\ldots t\}$.  Then $C$ is a
split extension of $B$ by an ideal $M =
\gen{\alpha_1,\alpha_2,\ldots,\alpha_t}$, generated by arrows such
that $\alpha_i$ is antiparallel to $\rho_i$ for each $i=1,\ldots,t$.
\end{prop}

\begin{proof}
Let $\Gamma=\{\alpha_1,\ldots,\alpha_t\}$ be an admissible cut of
$Q_C$ such that $B=C/\gen{\Gamma}$. Notice that for each subquiver
$\Sigma\simeq G(a,b)$ of $Q_C$ either $\eta\in \Gamma$ or
$\alpha\colon v_i\ra v_{i+1}$ and $\beta\colon v_j'\ra v_{j+1}'$ belong both to
$\Gamma$ (for some $i,\,j$). This shows that in each minimal relation
$\sigma=\sum_{j=1}^N c_j \sigma_j$ (where $\sigma_j$ are parallel paths
and $c_j\neq 0$ coefficients) defining the ideal $I_C$ we have that if
$\sigma_j\in\gen{\Gamma}$ for some $j$ then $\sigma_j\in\gen{\Gamma}$
for all $j$ and consequently $\sigma\in\gen{\Gamma}$. Hence by
\cite[Thm. 2.5]{ACT} we know that $C$ is the split extension of $B$ by the ideal
$\gen{\Gamma}$.
\end{proof}

\begin{remark}
  By the above the arrows in the admissible cut $\Gamma$ are in
  one-to-one correspondence to the relations defining $I_B$, with each
  arrow antiparallel to the corresponding relation.  Hence if $\gldim
  B\leq 2$ then the quiver of $C$ is precisely the quiver of $B$ with
  arrows added antiparallel to the relations in $I_B$. Thus, according
  to the description of the quiver of the relation extension given in
  \cite[Th. 2.6]{ABS}, if $\gldim B\leq 2$ the quiver of $C$ coincides
  with the quiver of $\mathcal R (B)$.
\end{remark}

\subsection{Strongly simple connectedness}

We refer to Section \ref{sec:sc} and the references cited there for the
definition of simple connectedness and strongly simple connectedness
of algebras.

\begin{lem}\label{lem:ssc}
  Let $B$ be a quotient by an admissible cut $\Gamma$ of a cluster
  tilted algebra $C$ of Dynkin type. Then $B$ is a strongly simply
  connected algebra.  
\end{lem}

\begin{proof}
  We know from Proposition \ref{prop:cyc-oriented} 
  that each chordless
  cycle in $Q_C$ is oriented and from Lemma \ref{lem:no-oriented} each
  chordless cycle in $Q_B$ is non-oriented.  We now proceed in steps.

  (1)\quad \emph{Each chordless cycle in $Q_B$ is non-oriented and
    obtained from a subquiver of $Q_C$ which is isomorphic to $G(a,b)$
    (for some $a$ and $b$) by removing the arrow corresponding to
    $\eta$.}

  Indeed, let $\Sigma\colon v_1\LL v_2 \LL \cdots \LL v_t \LL v_1$ be a
  chordless cycle in $Q_B$. If $\Sigma$ is oriented then $\Sigma$ can
  not be chordless in $Q_C$ since $B$ is the quotient by an admissible
  cut. If $\Sigma$ is non-oriented then by (1) it can also not be
  chordless in $Q_C$. So in any case there exists a chord $v_i\LL v_j$
  for some $i\not\equiv j\pm 1\ (\mod t)$.  After reordering, we can
  assume $i=1$ and take $j>1$ minimal such that a chord $\eta_1\colon v_1\LL
  v_j$ exists. Then $\Sigma_1\colon v_1\LL v_2\LL \cdots\LL v_j\LL v_1$ is a
  chordless cycle in $Q_C$ and therefore oriented. If we assume that
  $\Sigma_2\colon v_1\LL v_j\LL v_{j+1}\LL\cdots v_t\LL v_1$ is not a
  chordless cycle in $Q_C$ then there exists a chord $\eta_2\colon v_l\LL
  v_h$ for some $j\leq l<h-1\leq t$ (where $v_{t+1}\colon=v_1$) and if we
  take $l\geq j$ minimal and $h\leq t+1$ maximal then
  $\Sigma'\colon v_1\xLL{\eta_1} v_j\LL \cdots \LL v_l\xLL{\eta_2}
  v_h\LL\cdots \LL v_t\LL v_1$ is a chordless (and therefore oriented)
  cycle in $Q_C$ with two arrows $\eta_1$ and $\eta_2$ belonging to
  the admissible cut, a contradiction. This shows that $\Sigma_2$ is
  also oriented and therefore (1) holds.

  (2)\quad \emph{The quiver $Q_B$ is directed, that is, it does not
    contain an oriented cycle.}

  Assume by contradiction that an oriented cycle $\Gamma$ exists
  in $Q_B$ and suppose that $\Sigma$ is minimal with respect to the
  number of vertices. By (1) the cycle $\Sigma$ is not chordless in
  $Q_B$. This chord divides $\Sigma$ into two smaller cycles, one of
  them necessarily is oriented, in contradiction to the minimality of
  $\Sigma$.  

  (3)\quad \emph{The algebra is strongly simply connected.}

  Using (1) and (2) it is easy to see that the $(Q_B,I_B)$ is its own
  universal cover, in the sense of \cite{M-JAP}. Therefore by
  \cite[Thm~4.2]{M-JAP} the algebra $B$ is simply connected. Since $C$
  is of Dynkin type, then by \cite[Prop.1.2]{BMR} algebra $C$ and
  hence $B$ is of finite representation type and therefore by Remark
  \ref{rem:sc=ssc} the algebra $B$ is strongly simply connected.
\end{proof}

\subsection{Behaviour of the quadratic form}

For a definition of the quadratic forms $\chi_B$ and $q_B$ associated
to an algebra $B$ we refer to Section \ref{sec:qf} and the references
cited there. 

\begin{prop}
  \label{prop:qB-positive}
Let $B$ be a quotient by an admissible cut of a cluster-tilted algebra $C$ of
Dynkin type such that gldim $B \leq 2$. Then $q_B$ is positive
definite.
\end{prop}

\begin{proof}
  Since $C$ is mutation equivalent to a Dynkin diagram, we know by
  \cite{BGZ} that the quiver $Q_C$ admits a positive definite
  quasi-Cartan companion $A_C$. By Remark~\ref{rem:q=chi} it suffices
  thus to show that the quasi-Cartan matrix $A$ defined by the homological
  from $\chi_B$ is equivalent to $A_C$.

  It follows from Proposition~\ref{prop:fund} that 
  $$
  q_B(x)=
  \sum_{i=1}^n x_i^2 
  -\sum_{\alpha\in (Q_B)_1} x_{s(\alpha)} x_{t(\alpha)}
  +\sum_{\gamma\in\Gamma} x_{s(r_\gamma)} x_{t(r_\gamma)}.
  $$
  Therefore, the quasi-Cartan matrix $A$ defined by $q_B(x)=x^\top
  A x$ satisfies the property that $|A_{ij}|$ equals the number of
  arrows or relations (in either direction) in $B$ between the vertices $i$
  and $j$ or equivalently the number of arrows (in either direction)
  in $Q_C$. This shows that $A$ is quasi-Cartan companion of $Q_C$.
  Since $\Gamma$ is an admissible cut in $Q_C$, in each
  oriented cycle of $Q_C$ there is precisely one arrow $i\ra j$ for
  which $A_{ij}=1$ and for all other arrows $i\ra j$ in the same cycle
  we have $A_{ij}=-1$. Therefore $A$ satisfies the sign condition in
  \cite[Prop.~1.4]{BGZ} and by \cite[Prop.~1.5]{BGZ} the two matrices
  $A$ and $A_C$ are equivalent.   
\end{proof}

\subsection{Main result on admissible cuts}

We now have now gathered sufficient information on admissible cuts to
be able to prove the main result on admissible cuts for cluster-tilted
algebras of Dynkin type.

\begin{thm}\label{thm:main-cut}
Let $B$ be a quotient by an admissible cut of a cluster-tilted algebra
$C$ of Dynkin type $Q$. If $\gldim B \leq 2$ then $B$ is iterated
tilted of Dynkin type $Q$.
\end{thm}

\begin{proof}
  By Proposition~\ref{prop:qB-positive} the geometric form $q_B$ of
  $B$ is positive definite and by Lemma~\ref{lem:ssc} the algebra $B$
  is strongly simply connected. It follows thus from \cite{AS1} that
  $B$ is iterated tilted of Dynkin type.
\end{proof}

\begin{remark}
The following example shows that this result can not be extended
to cluster-tilted of type $\widetilde{\An}_n$. Let $B=\field Q_B/I_B$
where $Q_B$ is as depicted below on the left hand side and $I_B$ is
generated by the relation $\beta\alpha$. We indicated this below the
quiver $Q_B$. In the middle column the quiver and relations of
$\ealg(B)$ are shown. Observe that $\{\beta\}$ is an admissible cut of
the quiver of $\ealg(B)$. Finally on the left hand side you can see the
quotient of $\ealg(B)$ by the admissible cut $\{\beta\}$.

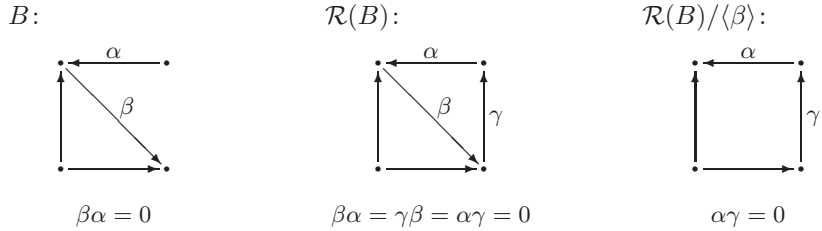
\begin{center}
  \begin{picture}(280,90)
    \put(0,20){
      \put(-20,55){$B\colon$}
      \multiput(0,0)(40,0){2}{\circle*{2}}
      \multiput(0,40)(40,0){2}{\circle*{2}}
      \put(3,0){\vector(1,0){34}}
      \put(0,3){\vector(0,1){34}}
      \put(2,38){\vector(1,-1){36}}
      \put(37,40){\vector(-1,0){34}}
      \put(25,24){\HVCenter{\small $\beta$}}
      \put(20,42){\HBCenter{\small $\alpha$}}      
      \put(20,-20){\HBCenter{\small $\beta\alpha=0$}}
    }
    \put(120,20){
      \put(-20,55){$\ealg(B)\colon$}
      \multiput(0,0)(40,0){2}{\circle*{2}}
      \multiput(0,40)(40,0){2}{\circle*{2}}
      \put(3,0){\vector(1,0){34}}
      \put(0,3){\vector(0,1){34}}
      \put(2,38){\vector(1,-1){36}}
      \put(37,40){\vector(-1,0){34}}
      \put(40,3){\vector(0,1){34}}
      \put(25,24){\HVCenter{\small $\beta$}}
      \put(20,42){\HBCenter{\small $\alpha$}}      
      \put(42,20){\LVCenter{\small $\gamma$}}
      \put(20,-20){\HBCenter{\small $\beta\alpha=\gamma\beta=\alpha\gamma=0$}}
    }
    \put(240,20){
      \put(-20,55){$\ealg(B)/\langle\beta\rangle\colon$}
      \multiput(0,0)(40,0){2}{\circle*{2}}
      \multiput(0,40)(40,0){2}{\circle*{2}}
      \put(3,0){\vector(1,0){34}}
      \put(0,3){\vector(0,1){34}}
      \put(37,40){\vector(-1,0){34}}
      \put(40,3){\vector(0,1){34}}
      \put(20,42){\HBCenter{\small $\alpha$}}      
      \put(42,20){\LVCenter{\small $\gamma$}}
      \put(20,-20){\HBCenter{\small $\alpha\gamma=0$}}
    }
  \end{picture}
\end{center}
Notice that $B$ is a tilted algebra of type $\widetilde{\An}_3$
and that $\gldim B \leq 2$ and hence $C=\ealg(B)$ is a cluster-tilted
algebra of type $\widetilde{\An}_3$, but the quotient
$B'=\ealg(B)/\langle\beta\rangle$ is not iterated tilted of any type
as shows the following argument.  Assume that $B'$ is an iterated
tilted algebra. Then the quiver of $\ealg(B')$ is isomorphic to
$Q_{\ealg(B)}=Q_C$. But by \cite[Thm. 2.3]{BORS} there is a unique
cluster-tilted algebra with quiver $Q_{C}$ and consequently by Theorem
\ref{thm:split} the algebra $B'$ is iterated tilted of type
$\widetilde{\An}_3$. But this constradicts the description in
\cite{AS2} of iterated tilted algebras of type $\widetilde{\An}_n$,
where it is shown that in a non-oriented cycle there must be as many
relatio s in clockwise orientations as there are relations in
counter-clockwise orientation.
\end{remark}

We prove now the main result of this section.

\renewcommand{\thedefthm}{\ref{thm:iff-cut}}
\begin{defthm}
  An algebra $B$ with $\gldim B\leq 2$ is iterated tilted of Dynkin type $Q$
  if and only if it is the quotient of a cluster-tilted algebra of type $Q$
  by an admissible cut.
\end{defthm}

\begin{proof}
  If $C$ is a cluster-tilted algebra of Dynkin type $Q$ then by
  Theorem \ref{thm:main-cut}, each quotient $B$ of $C$ by a admissible
  cut is an iterated tilted algebra with $\gldim B\leq 2$.

  Conversely, if $B$ is an iterated tilted algebra of Dynkin type $Q$
  with $\gldim B\leq 2$ then by Proposition \ref{prop:cut-exists} the
  algebra $B$ is a quotient of the relation-extension $\ealg(B)$ by an
  admissible cut $\Delta$. By \ref{rem:B-is-cut-of-R(B)-bis}.(b) it
  follows that $B$ is also an admissible cut of $C$ by $\Delta$.
\end{proof}

\subsection{Characterization when $\ealg(B)\simeq C(B)$}
We now want to study the relationship between a cluster-tilted algebra
$C$, a quotient $B$ of $C$ by an admissible cut and its relation
extension $\ealg(B)$.

\begin{remark}
\label{rem:not-rel-ext}
The following example shows that $C$ is in general
not the relation extension of $B$. To abbreviate notation we indicated
by dotted arcs where the composition of two consecutive arrows is zero.
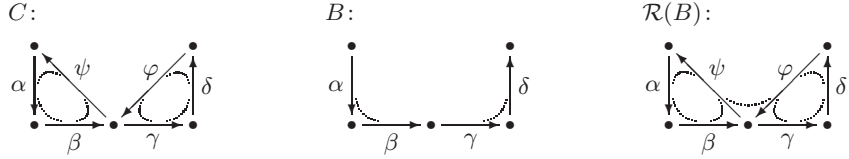
\begin{center}
  \begin{picture}(300,55)
    \put(0,8){
      \put(-10,40){\small $C\colon$}
      \multiput(0,0)(30,0){3}{\circle*{3}}
      \multiput(0,30)(60,0){2}{\circle*{3}}
      \put(0,26){\vector(0,-1){22}}
      \multiput(4,0)(30,0){2}{\vector(1,0){22}}
      \put(60,4){\vector(0,1){22}}
      \put(27,3){\vector(-1,1){24}}
      \put(57,27){\vector(-1,-1){24}}
      \qbezier[10](1.5,10)(3,3)(10,1.5)
      \qbezier[10](16,1.5)(22.5,3)(19,9.5)
      \qbezier[10](58.5,9.5)(57,3)(50,1.5)
      \qbezier[10](44,1.5)(37.5,3)(41,9.5)
      \qbezier[10](1.5,16)(3,22.5)(9.5,19)
      \qbezier[10](58.5,16)(57,22.5)(50,18.5)
      \put(-3,15){\RVCenter{\small $\alpha$}}
      \put(63,15){\LVCenter{\small $\delta$}}
      \put(15,-3){\HTCenter{\small $\beta$}}
      \put(45,-3){\HTCenter{\small $\gamma$}}
      \put(15,19){\small $\psi$}
      \put(47,19){\RBCenter{\small $\varphi$}}
    }
    \put(120,8){
      \put(-10,40){\small $B\colon$}
      \multiput(0,0)(30,0){3}{\circle*{3}}
      \multiput(0,30)(60,0){2}{\circle*{3}}
      \put(0,26){\vector(0,-1){22}}
      \multiput(4,0)(30,0){2}{\vector(1,0){22}}
      \put(60,4){\vector(0,1){22}}
      \qbezier[10](1.5,10)(3,3)(10,1.5)
      \qbezier[10](58.5,9.5)(57,3)(50,1.5)
      \put(-3,15){\RVCenter{\small $\alpha$}}
      \put(63,15){\LVCenter{\small $\delta$}}
      \put(15,-3){\HTCenter{\small $\beta$}}
      \put(45,-3){\HTCenter{\small $\gamma$}}
    }
    \put(240,8){
      \put(-10,40){\small $\ealg(B)\colon$}
      \multiput(0,0)(30,0){3}{\circle*{3}}
      \multiput(0,30)(60,0){2}{\circle*{3}}
      \put(0,26){\vector(0,-1){22}}
      \multiput(4,0)(30,0){2}{\vector(1,0){22}}
      \put(60,4){\vector(0,1){22}}
      \put(27,3){\vector(-1,1){24}}
      \put(57,27){\vector(-1,-1){24}}
      \qbezier[10](1.5,10)(3,3)(10,1.5)
      \qbezier[10](16,1.5)(22.5,3)(19,9.5)
      \qbezier[10](58.5,9.5)(57,3)(50,1.5)
      \qbezier[10](44,1.5)(37.5,3)(41,9.5)
      \qbezier[10](1.5,16)(3,22.5)(9.5,19)
      \qbezier[10](58.5,16)(57,22.5)(50,18.5)
      \qbezier[12](22,10)(30,5)(38,10)
      \put(-3,15){\RVCenter{\small $\alpha$}}
      \put(63,15){\LVCenter{\small $\delta$}}
      \put(15,-3){\HTCenter{\small $\beta$}}
      \put(45,-3){\HTCenter{\small $\gamma$}}
      \put(15,19){\small $\psi$}
      \put(47,19){\RBCenter{\small $\varphi$}}
    }
  \end{picture}
\end{center}

On the left hand side the cluster-tilted algebra $C$ is depicted. Then
$\Gamma=\{\varphi,\psi\}$ is an admissible cut and
$B=C/\gen{\Gamma}$ is as shown in the middle. On the right hand side 
we see the relation extension $\ealg(B)$. Note that
here we have $\psi\varphi=0$ whereas in $C$ this composition is
non-zero. 
\end{remark}

We describe now when $C(B)\simeq\ealg(B)$ for an iterated tilted
algebra $B$ such that $\gldim B\leq 2$. We know by Theorem
\ref{thm:split} that there is an exact sequence of algebra
homomorphisms $ B\ra C(B)\xrightarrow{\pi}\ealg(B)\ra B $ whose
composition is the identity map. Moreover, the kernel of $\pi$ is
contained in $\rad^2 C$.  Thus, we may assume that the
presentations of $C(B)$ and $\ealg(B)$ extend the presentation of
$B$, see section \ref{subsec:quivers}, and denote by $ \eta_1,
\dots, \eta_n$ the arrows of $Q _{C(B)} = Q_{\ealg (B)}$ which are not
arrows of $B$. Then by Propostion~\ref{prop:Ker}, we have $\Ker \pi =
\gen {\eta_1, \dots, \eta_n}^2_{C(B)}$, where the subscript
indicates that $\gen{\eta_1,\ldots,\eta_n}$
has to be considered as ideal of $C(B)$.

\begin{prop}
  \label{prop:char-r=c}
  Let $B$ be an iterated tilted algebra such that $\gldim B\leq 2$ and
  let $ \eta_1, \dots, \eta_n$ be as above. Then the following
  conditions are equivalent.
  \begin{itemize}
  \item[{\rm (a)}] $\ealg(B)\simeq C(B)$.

  \item[{\rm (b)}] $\gen {\eta_1, \dots, \eta_n}^2_{C(B)}=0$.

  \item[{\rm (c)}] $\eta_i\, \mu \,\eta_j = 0$ in $ C(B)$ for any path
    $\mu \in kQ_B$ and for all $1\leq i,j \leq n$.
  \end{itemize}

  If we assume moreover that $B$ is of Dynkin type then {\rm (a)},
    {\rm (b)}
    and {\rm (c)} are equivalent to the following condition.

  \begin{itemize}
  \item [{\rm (d)}] Let $\rho_1\colon a \rightarrow i$, $\rho_2\colon
    j \rightarrow b$ be minimal relations in $B$ such that there is a
    non-zero path $\mu\colon a \rightarrow b$ in $kQ_B$.  Then
      for $h=1$ or $h=2$ the following holds:
      there are paths $\mu_1, \mu_2$ such that $\mu =\mu_2
      \mu_1$, an arrow
      $\alpha_h$ and, in case $\rho_h$ is not a zero relation then
      there exists a path $\gamma_h$ not involving $\alpha_h$ (set
      $\gamma_h=0$ otherwise) such
      that $\rho_1= \alpha_1\mu_1 - \gamma_1$ or $\rho_2=
      \mu_2\alpha_2 - \gamma_2$ respectively. Furthermore, $\rho_h$ is the
      only minimal relation involving $\alpha_h$.
  \end{itemize}
\end{prop}

\begin{proof}
  Since $\Ker\pi = \gen {\eta_1, \dots, \eta_n}^2_{C(B)}$, the
  equivalence of (a) and (b) follows from the fact that $\gen {\eta_1,
    \dots, \eta_n}^2_{\ealg(B)}=0$. The equivalence of (b) and (c) is
  straightforward, so we only need to prove that (c) and (d) are
  equivalent in the Dynkin case.

  Thus we assume from now on that $B$ is of Dynkin type. Then
  $\{ \eta_1, \dots, \eta_n\}$ is an admissible cut of $C(B)$, by
  Proposition~\ref{prop:cut-exists}.

  First assume that (c) holds, and consider $\rho_1, \mu$ and $\rho_2$
  as in (d). Then each relation $\rho_i$ corresponds to an arrow
  $\eta_{k_i}$. We may assume that $k_h = h$ and by (c) we have that
  $\eta_{2} \mu \, \eta_1 =0$ in $C(B)$. If this relation is minimal
  we know from Theorem~\ref{thm:Dynkin-rel} that there exists an arrow
  $\alpha$ so that $\alpha \eta_{2} \mu \, \eta_1$ is a chordless
  oriented cycle, contradicting that $\{\eta_1, \dots, \eta_n\} $ is
  an admissible cut of $C(B)$. Therefore the relation $\eta_{2} \mu \,
  \eta_1 =0$ is not minimal, and hence there are paths $\mu_1, \mu_2$
  such that $\mu = \mu_2 \mu_1$ and either $\mu_1\eta_1$ or
  $\eta_2\mu_2$ is a minimal zero relation in $C(B)$.  In the first
  case, by Theorem~\ref {thm:Dynkin-rel}, there is an arrow $\alpha_1$
  such that $\mu_1\eta_1 \alpha_1$ is an oriented chordless cycle in
  $C(B)$, and $\alpha_1$ is not contained in any other chordless cycle
  in $C(B)$. Then $\alpha_1\mu_1$ is a shortest path antiparallel to
  $\eta_1$ and the statement follows from Theorem~\ref{thm:Dynkin-rel}
  using that $\rho_1$ is the relation antiparallel to $\eta_1$. The
  case when $\eta_2\mu_2$ is a minimal zero relation can be
  handled in a similar way, so (d) holds.

  Now assume that (d) holds and consider a path $\eta_s \mu \eta_r$
  with $\mu\in kQ_B$.  Consider the minimal relations $\rho_1, \rho_2$
  in $I_B$ antiparallel to $\eta_r, \eta_s$ respectively and let
  $h$ and $\alpha_h$, $\mu_1$, $\mu_2$, $\gamma_h$ be as in
  (d).  If $h=1$, that is, $\rho_1 = \alpha_1\mu_1
  -\gamma_1$ then $\eta_r$ is antiparallel to
  $\alpha_1\mu_1$, since $\eta_r$ is antiparallel to
  $\rho_1$. Then $\alpha_1\mu_1\eta_r$ is a chordless cycle in
  $C(B)$ and from the description of the relations in
  Theorem~\ref{thm:Dynkin-rel} we obtain that $\mu_1\eta_r=0$,
  since $\alpha_1$ is involved in a unique minimal relation. Thus
  $\eta_s \mu \eta_r= \eta_s \mu_2 \mu_1 \eta_r=0$ in this
  case. The same argument applies in the other case, proving (c).
\end{proof}

When the iterated tilted algebra $B$ is given by its quiver and
relations and is of Dynkin type then (d) provides an easy way to
determine if $\ealg (B)$ and $C(B)$ are isomorphic. For example, if
two minimal relations of $B$ are consecutive then (d) is not
satisfied. Using this one readily verifies that $\ealg (B_i)$ and $
C(B_i)$ are not isomorphic for the algebras $B_0$, $B_1$ and $B_2$ of
Example~\ref{D8}, and also that $\ealg(B) \not\simeq C(B)$ in
Remark~\ref{rem:not-rel-ext}.



\begin{thebibliography}{99}

\bibitem{A}{C.~Amiot: \sl{Cluster categories for algebras of global
      dimension 2 and quivers with potencials}. {\tt arXiv.math0805.1035}.}

\bibitem{ABS}{I.~Assem, Th.~Br\"ustle, R.~Schiffler: \sl{Cluster
    tilted algebras as trivial extensions}, Preprint {\tt
    arXiv.math601537}.} 

\bibitem{ACT}{I.~Assem, F.~Coelho, S.~Trepode: \sl{The bound quiver of
      a split extension}, To appear in J. Algebra Appl.}

\bibitem{AHR}{
I.~Assem, D.\ Happel and O.\ Rold\'an: {\sl Representation-finite trivial extension algebras.} 
J. Pure Appl. Algebra 33 
(1984), 235--242.
}


\bibitem{AS1}{
I.~Assem, A.~Skowro\'nski:
\emph{Quadratic forms and iterated tilted algebras.}\/
J.\ Algebra  128  (1990),  no.\ 1, 55--85.
} 

\bibitem{AS2}{I.~Assem, A.~ Skowro\'nski: {\sl Iterated tilted
    algebras of type $\tilde{A}_n$.} Math. Z. 195 (1987), 269--290.}

\bibitem{ASS}{I.~Assem, A.~Skowro\'nski, D.~Simson: {\sl Elements of the
    Representation Theory of Associative Algebras: Techniques of
    Representation Theory; Vol. 1.}, 400 p. 2006 Cambridge U.P. ISBN:
  0-521-58631-3b}

\bibitem{BGZ}{M.~Barot, Ch.~Geiss, A.~Zelevinsky: \sl{Cluster algebras of finite type and positive symmetrizable matrices}. Journal of the London
Mathematical Society 73 (2006), 545--564.}

\bibitem{BG}{O. Bretscher, P. Gabriel: \sl{The standard from of a
    representation-finite algebra}. Bull. Soc. Math. France 111 (1983),
  no. 1, 21--40.}

\bibitem{BMRRT}{
A.~Buan, R.~Marsh, M.~Reineke, I.~Reiten and G.~Todorov: {\it
Tilting theory and cluster combinatorics}. Adv. Math. 204
(2006), no. 2, 572--618.}

\bibitem{BMR}{A.~Buan, R.~Marsh, I.~Reiten: \sl{Cluster-tilted
      algebras of finite representation type}.  J.~Algebra
    306  (2006),  no. 2, 412--431.}

\bibitem{BMR2}{A.~Buan, R.~Marsh, I.~Reiten: \sl{Cluster-tilted
      algebras}. Trans. Amer. Math. Soc., 239 (2007), no. 1, 323--332.}

\bibitem{BMR3}{A.~Buan, R.~Marsh, I.~Reiten: \sl{Cluster mutation via
      quiver representations}. Coment. Math. Helvetici 83 (2008), no. 1, 143--177.} 

\bibitem{BMRT}{A.~Buan, R.~Marsh, I.~Reiten, G.~Todorov: \sl{Clusters
    and seeds in acyclic cluster algebras}.
  Proc.\ Amer.\ Math.\ Soc.\ 135 (2007), no.~10, 3049--3060.}

\bibitem{BORS}{A.~Buan, O.~Iyama, I.~Reiten, D.~Smith: {\sl Mutation
    of cluster-tilting objects and potentials}. Preprint 2008. {\tt
    arXiv.math:0804.3813v2}.}

\bibitem{CK2}{F.\ Caldero, B.\ Keller: 
{\sl From triangulated categories to cluster algebras
II.} Ann. Sci. Ecole. Norm. Sup. 4eme. serie (2006) 983--1005.}

\bibitem{F}{E.~Fern\'andez : {\sl Extensiones triviales y \'algebras inclinadas iteradas}. Ph.D. Thesis, Universidad Nacional del Sur, Argentina (1999).}

\bibitem{FP}{E.~Fern\'andez, M. I.  Platzeck: {\sl Presentations of
      trivial extensions of finite dimensional algebras and a theorem
      of Sheila Brenner}. J. \ Algebra 249 (2002) 326--344.}

\bibitem{FZ1}{S.~Fomin, A.~Zelevinsky: {\sl Cluster algebras I:
      Foundations}. J. Amer. Math. Soc. 15 (2002), no. 2, 497--529.}

\bibitem{FZ2}{S.~Fomin, A.~Zelevinsky: \sl{Cluster algebras II: Finite type
    classification}. Invent. Math. 154 (2003), no 1. 63--121.}

\bibitem{Ha}{D.\ Happel: {\sl Triangulated categories in the representation theory of
  finite-dimensional algebras.}
London Mathematical Society Lecture Note Series, 119.
Cambridge University Press, Cambridge, 1988.}


\bibitem{I}{
O.\ Iyama: {\sl Cluster tilting for higher Auslander Algebras.} ArXiv: 0809.4897
}

\bibitem{HW}{
D.~Hughes and J.\ Waschb\"usch: {\sl Trivial extensions of tilted algebras.} 
Proc. London Math. Soc. 46
(1983), 347--364.
}


\bibitem{Ke}{
B.~Keller:
{\sl On triangulated orbit categories.}
Doc. Math. 10 (2005), 551--581.
}


\bibitem{KZ}{
S.~K\"onig, A.~Zimmermann: {\sl Derived equivalences for group
rings}, (with contributions by Bernhard Keller, Markus Linckelmann, Jeremy
Rickard and Raphael Rouquier) 255 pages, Springer Lecture Notes in
Mathematics 1685 (1998).}


\bibitem{M-JAP}{R.~Mart\'{\i}nez-Villa, J.~A.~de la Pe\~na: \sl{The
      universal cover of a quiver with relations.\/}
    J. Pure Appl. Algebra  30  (1983),  no. 3, 277--292.} 

\bibitem{RvB}{I. Reiten, M. Van den Berg: \sl{Grothendieck groups and
    tilting objects}. Algebr. Represent. Theory 4 (2001), no. 1,
  1--21.}

\bibitem{Ric}{J.~Rickard:
{\sl Morita theory for derived categories.}
J. London Math. Soc.,  {\bf 39},
(1989) 436--456.}



\bibitem{Ri}{C.~M.~Ringel:
{\sl Tame algebras and integral quadratic forms.}
Lecture Notes in Mathematics {\bf 1099},
Springer-Verlag, Berlin, 1984.}

\bibitem{Sko}{
A.~Skowro\'nski: \sl{Simply connected algebras and Hochschild
  cohomologies.\/}  Representations of algebras (Ottawa, ON, 1992),
431--447, CMS Conf. Proc., 14, Amer. Math. Soc., Providence, RI,
1993.} 


\end{thebibliography}
\end{document}